\documentclass[12pt]{amsart}
\usepackage{fullpage}
\usepackage{amsfonts,amscd}
\usepackage{amssymb}
\usepackage{url}
\usepackage{graphicx}
\usepackage{inputenc}
\usepackage[english]{babel}
\theoremstyle{plain}
\newtheorem{theorem}                {Theorem}      [section]
\newtheorem{proposition}  [theorem]  {Proposition}

\newtheorem{lemma}        [theorem]  {Lemma}

\theoremstyle{definition}
\newtheorem{example}      [theorem]  {Example}
\newtheorem{remark}       [theorem]  {Remark}
\newtheorem{definition}   [theorem]  {Definition}

\numberwithin{equation}{section}
\def \t{\mbox{${\mathbb T}$}}
\def \R{{\mathbb R}}
\def \s{{\mathbb S}}
\def \z{{\mathbb Z}}
\def \n{{\mathbb N}}

\def \C{{\mathbb C}}

\usepackage{color}

\def \link {~}
\def \1 {\`}

\DeclareMathOperator{\trace}{trace}

\def \1{\mbox{${\mathbf 1}$}}
\def \r{\mbox{${\mathbb R}$}}
\def \s{\mbox{${\mathbb S}$}}

\def \n{\mbox{${\mathbb N}$}}

\DeclareMathOperator{\Index}{Index}
\DeclareMathOperator{\nul}{Nullity}

\numberwithin{equation}{section}

\begin{document}

\title[]{Index and nullity of proper biharmonic maps in spheres}

\author{S.~Montaldo}
\address{Universit\`a degli Studi di Cagliari\\
Dipartimento di Matematica e Informatica\\
Via Ospedale 72\\
09124 Cagliari, Italia}
\email{montaldo@unica.it}

\author{C.~Oniciuc}
\address{Faculty of Mathematics\\ ``Al.I. Cuza'' University of Iasi\\
Bd. Carol I no. 11 \\
700506 Iasi, ROMANIA}
\email{oniciucc@uaic.ro}

\author{A.~Ratto}
\address{Universit\`a degli Studi di Cagliari\\
Dipartimento di Matematica e Informatica\\
Via Ospedale 72\\
09124 Cagliari, Italia}
\email{rattoa@unica.it}

\begin{abstract}
In recent years, the study of the bienergy functional has attracted the attention of a large community of researchers, but there are not many examples where the second variation of this functional has been thoroughly studied. We shall focus on this problem and, in particular, we shall compute the exact index and nullity of some known examples of proper biharmonic maps. Moreover, we shall analyse a case where the domain is not compact. More precisely, we shall prove that a large family of proper biharmonic maps $\varphi: \R \to \s^2$ is strictly stable with respect to compactly supported variations.
In general, the computations involved in this type of problems are very long. For this reason, we shall also define and apply to specific examples a suitable notion of index and nullity with respect to equivariant variations. 
\end{abstract}



\thanks{The first and the last authors  were supported by Fondazione di Sardegna (project STAGE) and Regione Autonoma della Sardegna (Project KASBA); the second author was supported by a project funded by the Ministry of
Research and Innovation within Program 1 - Development of the national RD
system, Subprogram 1.2 - Institutional Performance - RDI excellence
funding projects, Contract no. 34PFE/19.10.2018.}
\maketitle

\section{Introduction}\label{intro}
{\it Harmonic maps} are the critical points of the {\em energy functional}
\begin{equation}\label{energia}
E(\varphi)=\frac{1}{2}\int_{M}\,|d\varphi|^2\,dv_M \,\, ,
\end{equation}
where $\varphi:M\to N$ is a smooth map from a compact Riemannian
manifold $(M,g)$ to a Riemannian
manifold $(N,h)$. In particular, $\varphi$ is harmonic if it is a solution of the Euler-Lagrange system of equations associated to \eqref{energia}, i.e.
\begin{equation}\label{harmonicityequation}
  - d^* d \varphi =   {\trace} \, \nabla d \varphi =0 \,\, .
\end{equation}
The left member of \eqref{harmonicityequation} is a vector field along the map $\varphi$ or, equivalently, a section of the pull-back bundle $\varphi^{-1} TN$: it is called {\em tension field} and denoted $\tau (\varphi)$. In addition, we recall that, if $\varphi$ is an \textit{isometric immersion}, then $\varphi$ is a harmonic map if and only if it defines a \textit{minimal submanifold} of $N$ (see \cite{EL1, EL83} for background).

A related topic of growing interest is the study of {\it biharmonic maps}: these maps, which provide a natural generalisation of harmonic maps, are the critical points of the {\it bienergy functional} (as suggested in \cite{EL83}, \cite{ES})
\begin{equation}\label{bienergia}
    E_2(\varphi)=\frac{1}{2}\int_{M}\,|d^*d\varphi|^2\,dv_M=\frac{1}{2}\int_{M}\,|\tau(\varphi)|^2\,dv_M\,\, .
\end{equation}
There have been extensive studies on biharmonic maps. We refer to \cite{Chen, Jiang, SMCO, Ou} for an introduction to this topic and to \cite{MOR2, Mont-Ratto2, Mont-Ratto3, Mont-Ratto6} for a collection of examples which shall be studied in this paper in the context of second variation.
We observe that, obviously, any harmonic map is trivially biharmonic and an absolute minimum for the bienergy. Therefore, we say that a biharmonic map is {\it proper} if it is not harmonic and, similarly, a biharmonic isometric immersion is {\it proper} if it is not minimal.
As a general fact, when the ambient has nonpositive sectional curvature there are several results which assert that, under suitable conditions, a biharmonic submanifold is minimal, but the Chen conjecture that any biharmonic submanifold of $\R^n$ must be minimal is still open (see \cite{Chen, Chen2}). 

The aim of this paper is to compute the index and the nullity of certain biharmonic maps.
It shall be clear from our analysis that, in general, despite the simplicity of the involved maps, this is a hudge task (for this reason some of the computations have also been  checked with the aid of Mathematica$^{\footnotesize \textregistered}$). Therefore, in some cases, we shall focus on reduced index and nullity (i.e., index and nullity which arise from the restriction to equivariant variations).

Now we want to prepare the ground to state our main results. To this purpose, first of all we need to explain some basic facts about the iterated Jacobi operator $I_2(V)$ and the definition of index and nullity. More specifically, let $\varphi:M\to N$ be a biharmonic map between two Riemannian manifolds $(M,g)$, $(N,h)$. We shall consider a two-parameter smooth variation $\left \{ \varphi_{t,s} \right \}$ $(-\varepsilon <t,s < \varepsilon,\,\varphi_{0,0}=\varphi)$ and denote by $V,W$ its associated vector fields:
\begin{align}\label{V-W}
& V(x)= \left . \frac{d}{dt}\right |_{t=0} \varphi_{t,0}  \in T_{\varphi(x)}N \\ \nonumber
& W(x)= \left . \frac{d}{ds}\right |_{s=0} \varphi_{0,s} \in T_{\varphi(x)}N\,.
\end{align}
Note that $V$ and $W$ are sections of $\varphi^{-1}TN$. The {\it Hessian} of the bienergy functional $E_2$ at its critical point $\varphi$ is defined by
\begin{equation}\label{Hessian-definition}
H(E_2)_\varphi (V,W)= \left . \frac{\partial^2}{\partial t \partial s}\right |_{(t,s)=(0,0)}  E_2 (\varphi_{t,s}) \, .
\end{equation}
The following theorem was obtained by Jiang and translated by Urakawa \cite{Jiang}:
\begin{theorem}\label{Hessian-Theorem} Let $\varphi:M\to N$ be a biharmonic map between two Riemannian manifolds $(M,g)$ and $(N,h)$, where $M$ is compact. Then the Hessian of the bienergy functional $E_2$ at a critical point $\varphi$ is given by
\begin{equation}\label{Operator-Ir}
H(E_2)_\varphi (V,W)= \int_M \langle I_2(V),W \rangle \, dv_M \,\,,
\end{equation}
where $I_2 \,:\mathcal{C}\left(\varphi^{-1} TN\right) \to \mathcal{C}\left(\varphi^{-1} TN\right)$ is a semilinear elliptic operator of order $4$.
\end{theorem}
Now we want to give an explicit description of the operator $I_2$. To this purpose, let $\nabla^M, \nabla^N$ and $\nabla^{\varphi}$ be the induced connections on the bundles $TM, TN$ and $\varphi ^{-1}TN$ respectively. Then the \textit{rough Laplacian} on sections of $\varphi^{-1} TN$, denoted by $\overline{\Delta}$, is defined by
\begin{equation}\label{roughlaplacian}
    \overline{\Delta}=d^* d =-\sum_{i=1}^m\Big\{\nabla^{\varphi}_{e_i}
    \nabla^{\varphi}_{e_i}-\nabla^{\varphi}_
    {\nabla^M_{e_i}e_i}\Big\}\,\,,
\end{equation}
where $\{e_i\}_{i=1}^m$ is a local orthonormal frame field tangent to $M$. 
In the present paper, we shall only need the explicit expression of $I_2(V)$ in the case that the target manifold is $\s^n$. This relevant formula, which was first given in \cite{Onic} and can be deduced from a general formula in \cite{Jiang}, is the following:
\begin{eqnarray}\label{I2-general-case}
\nonumber
I_2(V)&=&\overline{\Delta}^2 V+\overline{\Delta}\left ( {\rm trace}\langle V,d\varphi \cdot \rangle d\varphi \cdot - |d\varphi|^2\,V \right)+2\langle d\tau(\varphi),d\varphi \rangle V+|\tau(\varphi)|^2 V\\ \nonumber
&&-2\,{\rm trace}\langle V,d\tau(\varphi) \cdot \rangle d\varphi \cdot-2 \,{\rm trace} \langle \tau(\varphi),dV \cdot \rangle d\varphi \cdot - \langle \tau(\varphi),V \rangle \tau(\varphi)\\ 
&&+{\rm trace} \langle d\varphi \cdot,\overline{\Delta}V \rangle d\varphi \cdot +{\rm trace}\langle d\varphi \cdot,\left ( {\rm trace}\langle
V,d\varphi \cdot \rangle d\varphi \cdot \right ) \rangle d\varphi \cdot -2 |d\varphi|^2\, {\rm trace}\langle d\varphi \cdot,V \rangle
 d\varphi\cdot \\ \nonumber
&&+2 \langle dV,d\varphi\rangle \tau(\varphi) -|d\varphi|^2 \,\overline{\Delta}V+|d\varphi|^4 V\,, \nonumber
\end{eqnarray}
where $\cdot$ denotes trace with respect to a local orthonormal frame field on $M$. Next, it is important to recall from the general theory that, since $M$ is compact, the spectrum
\begin{equation}
\lambda_1 < \lambda_2 < \ldots < \lambda_i < \ldots
\end{equation}
of the iterated Jacobi operator $I_2(V)$ is \textit{discrete} and tends to $+\infty$ as $i$ tends to $+ \infty$. We denote by $\mathcal{V}_i$ the eigenspace associated to the eigenvalue $\lambda_i$. Then we define
\begin{equation}\label{Index-definition}
{\rm Index}(\varphi) = \sum_{\lambda_i <0} \dim(\mathcal{V}_i)\,.
\end{equation}
The nullity of $\varphi$ is defined as
\begin{equation}\label{Nullity-definition}
{\rm Nullity}(\varphi) =  \dim \left \{ V \in\mathcal{C}\left(\varphi^{-1} TN\right) \, : \, I_2(V)=0\right \} \,.
\end{equation}
We say that a map $\varphi: M\to N$ is {\it stable} if ${\rm Index}(\varphi)=0$.

The index and the nullity of certain proper biharmonic maps have been computed in \cite {BFO,LO, TAMS} where, apart from one case, only estimates have been produced.  In this paper we continue this program of study of the second variation of the bienergy and now we are in the right position to describe the specific examples that we shall investigate: each of them contains a short description of the biharmonic maps under consideration and the corresponding result concerning their exact index and nullity. In the first examples the domain of the map is a flat torus or a circle, for which the full description of its spectrum is well known, and the pull-back bundle of the map is parallelizable. Moreover, in the first example, where the domain is the flat torus $\mathbb{T}^2$, we shall also give an explicit geometric description of the space where the Hessian is negative definite. In the last example we shall consider a case where the domain is \textit{not} compact: in this context it is meaningful to study \textit{stability with respect to compactly supported variations}. In particular, we shall prove the existence of a large family of \textit{strictly stable} proper biharmonic maps $\varphi: \R \to \s^2$. The proofs of the results shall be given in Section\link\ref{proofs}. Finally, in the last section we shall define and study a \textit{reduced} index and nullity.

We shall now give a detailed description of the results.

\begin{example}\label{example-mapstoro-to-sfera}
We write the flat 2-torus ${\mathbb T}^2$ as
\begin{equation}\label{toro}
{\mathbb T}^2 = \left ( \s ^1 \times \s^1, d\gamma^2+ d\vartheta^2 \right ) \,\, , \qquad 0 \leq \gamma, \vartheta \leq 2 \pi \,\,.
\end{equation}
Next, we describe the 2-sphere $\s^2$ by means of spherical coordinates:
\begin{equation}\label{2sfera}
\s^2 = \left ( \s^1 \times [0,\pi], \, \sin^2 \alpha \, \, dw^2+ d\alpha^2 \right ) \,\, , \qquad 0 \leq w \leq 2 \pi \,, \,\, 0 \leq \alpha\leq \pi \,\, .
\end{equation}
We embed $\s^2$ in the canonical way into $\R^3$ and we consider equivariant maps $\varphi_{k} : {\mathbb T}^2 \to \s^2$ of the following form:
\begin{equation}\label{equivdatoroasfera}
    \left ( \gamma, \, \vartheta \right ) \mapsto \left ( \sin \alpha(\vartheta) \cos(k\gamma), \sin \alpha(\vartheta) \sin(k\gamma),\cos  \alpha (\vartheta) \right ) \,\, ,
\end{equation}
where $k\in \z^*$ is a fixed integer and $\alpha(\vartheta)$ is a differentiable, periodic function of period equal to $2\pi$.
The condition of biharmonicity (see \cite{Mont-Ratto3}) for $\varphi_{k}$ reduces to:
\begin{equation}\label{biarmoniatorosfera}
    \alpha^{(4)} - {\alpha}'' \, \left [ 2 \, k^2 \, \cos (2 \alpha)\right ] + ({\alpha'})^2 \, \left [ 2 \, k^2 \, \sin (2 \alpha)\right ] + \, \frac{k^4}{2} \, \sin (2\alpha) \, \cos (2 \alpha) \, = \, 0 \,\, .
\end{equation}
In particular, \eqref{biarmoniatorosfera} admits the following constant solutions:
\begin{equation}\label{trivialsolutions}
    {\rm (i)}\,\, \alpha \, \equiv \ell \, \frac {\pi}{2} \,\,, {\rm where} \,\,  \ell =0,\,1,\,2  \, \, ; \qquad \rm {(ii)}\,\, \alpha\, \equiv \frac {\pi}{4} \,\,\, {\rm or}\, \,\,  \alpha \, \equiv \frac {3\, \pi}{4} \,\, .
\end{equation}
The solutions in \eqref{trivialsolutions}(i) are not interesting because they give rise to harmonic maps which are absolute minima for the bienergy. By contrast, the solutions in \eqref{trivialsolutions}(ii) represent {proper biharmonic maps} and here we study their second variation operator $I_2$. Despite the apparently simple structure of these critical points, the study of their index and nullity requires a rather accurate analysis. Since index and nullity are invariant with respect to composition with an isometry of either the domain or the target, it is not restrictive to assume that $k\in \n^*$ in \eqref{equivdatoroasfera} and $\alpha=\pi /4$ in \eqref{trivialsolutions}. We shall prove the following result:
\begin{theorem}\label{Spectrum-theorem-toro-sfera} Let $\varphi_{k} : {\mathbb T}^2 \to \s^2$ be the proper biharmonic map
\begin{equation}\label{equivdatoroasfera-bis*}
    \left ( \gamma, \, \vartheta \right ) \mapsto \left ( \sin (\alpha^*) \cos(k\gamma), \sin (\alpha^*) \sin(k\gamma),\cos  (\alpha^*) \right ) \,\, ,
\end{equation}
where $k\in \n^*$ and $\alpha^*= \pi /4$.
Then the eigenvalues of the second variation operator $I_2(V)$ associated to $\varphi_{k}$ can be parametrized by means of two parameters $m,n\in \n\times \n$ as follows:
\renewcommand{\arraystretch}{1.3}
\begin{equation*}
\begin{tabular}{ l r c l}
  If $m=n=0$: & && \\
 & $\mu_0$&$=$&$0$  \\
 & $\mu_1$&$=$&$-k^4$  \\
   If  $m \geq 1,\,n=0$:& &&  \\
  &$\lambda_{m,0}^+$ &$=$&$\frac{1}{2} \left(-k^4+5k^2 m^2+2 m^4+ \sqrt{k^8+2 k^6 m^2+k^4
m^4+32k^2 m^6 }\right)$ \\
&$\lambda_{m,0}^- $&$=$&$\frac{1}{2} \left(-k^4+5k^2 m^2+2 m^4- \sqrt{k^8+2 k^6 m^2+k^4
   m^4+32k^2 m^6 }\right)$\\
    If  $m=0,\,n \geq 1$:&  && \\
  &$\lambda_{0,n}^+$ &$=$&$n^2 (n^2+k^2)$ \\
&$\lambda_{0,n}^-$ &$=$&$n^4-k^4$\\
 If  $m,n \geq 1$:& &&  \\
  &$\lambda_{m,n}^+$ &$=$&$\frac{1}{2}\left(-k^4+k^2 \left(5 m^2+n^2\right)+2 \left(m^2+n^2\right)^2 \right .$ \\
  &&&$ \left . + \sqrt{k^8+2 k^6 \left(m^2+n^2\right)+k^4
   \left(m^2+n^2\right)^2+32 k^2 m^2 \left(m^2+n^2\right)^2}\right)$ \\
&$\lambda_{m,n}^-$ &$=$&$\frac{1}{2}\left(-k^4+k^2 \left(5 m^2+n^2\right)+2 \left(m^2+n^2\right)^2 \right .$\\
&&& $ \left . - \sqrt{k^8+2 k^6 \left(m^2+n^2\right)+k^4   \left(m^2+n^2\right)^2+32 k^2 m^2 \left(m^2+n^2\right)^2}\right)$
\end{tabular}
\end{equation*}
\renewcommand{\arraystretch}{1.}
Moreover, the multiplicities $\nu(\lambda)$ of these eigenvalues are:
\begin{equation}\label{eigenspaces-toro-sfera}
\begin{tabular}{c c c cc}
$\nu \left ( \mu_0\right )$&=&$\nu \left ( \mu_1\right )$&=&$1$ \\
$\nu \left ( \lambda_{m,0}^+\right )$&=&$\nu \left ( \lambda_{m,0}^-\right )$&=&$2$  \\
$\nu \left ( \lambda_{0,n}^+\right )$&=&$\nu \left ( \lambda_{0,n}^-\right )$&=&$2$  \\
$\nu \left ( \lambda_{m,n}^+\right )$&=&$\nu \left ( \lambda_{m,n}^-\right )$&=&$4$  \\
\end{tabular}
\end{equation}
\end{theorem}
Now, in order to state our next result, it is convenient to define, for $k \in \n^*$,
\begin{equation}\label{definizione-f(k)}
f(k)=\sharp \left \{[m,n]\in \n^* \times \n^* \,\,:\,\,\lambda_{m,n}^- <0  \right \}
\end{equation}
and
\begin{equation}\label{definizione-g(k)}
g(k)=\sharp \left \{[m,n]\in \n^* \times \n^* \,\,:\,\,\lambda_{m,n}^- =0  \right \}.
\end{equation}
We point out that it is not difficult to prove that $f(k)\leq k^2$ for all $k \in \n^*$ and that $f(k)$ does not admit a polynomial expression.
\begin{theorem}\label{Index-theorem-toro-sfera} Let $\varphi_{k} : {\mathbb T}^2 \to \s^2$ be the proper biharmonic map \eqref{equivdatoroasfera-bis*}. Then
\begin{align}
& {\rm Nullity}(\varphi_{k})=5+4\,g(k)  \\
& {\rm Index}(\varphi_{k})=1+4(k-1)+4\, f(k).
\end{align}
\end{theorem}

\begin{remark}\label{remark-stime-index-toro-sfera}
The eigenvalues in Theorem~\ref{Spectrum-theorem-toro-sfera} are not written in an increasing order. Moreover, we observe that some eigenvalues in Theorem~\ref{Spectrum-theorem-toro-sfera} may occur more than once. For instance, $\mu_0=\lambda_{k,0}^-=\lambda_{0,k}^-$.

The numerical values of the functions $f$ and $g$ in the statement of Theorem\link\ref{Index-theorem-toro-sfera} can be computed by means of a suitable computer algorithm. By way of example, we report some of them in the following table:
\begin{equation}\label{Table-index-toro-sfera}
\,\begin{array}{lllll}
k=1 &\quad \quad &{\rm Index}(\varphi_{k})=1 &\quad \quad& {\rm Nullity}(\varphi_{k})=5 \\ \nonumber
k=2 &\quad \quad &{\rm Index}(\varphi_{k})=13 &\quad \quad& {\rm Nullity}(\varphi_{k})=5 \\ \nonumber
k=3 &\quad \quad &{\rm Index}(\varphi_{k})=29 &\quad \quad& {\rm Nullity}(\varphi_{k})=5 \\ \nonumber
k=4 &\quad \quad &{\rm Index}(\varphi_{k})=57 &\quad \quad& {\rm Nullity}(\varphi_{k})=5 \\ \nonumber
k=5 &\quad \quad &{\rm Index}(\varphi_{k})=89 &\quad \quad& {\rm Nullity}(\varphi_{k})=5 \\ \nonumber
k=6 &\quad \quad &{\rm Index}(\varphi_{k})=129 &\quad \quad& {\rm Nullity}(\varphi_{k})=5 \\ \nonumber
k=7 &\quad \quad &{\rm Index}(\varphi_{k})=181 &\quad \quad& {\rm Nullity}(\varphi_{k})=5 \\ \nonumber
k=8 &\quad \quad &{\rm Index}(\varphi_{k})=233 &\quad \quad& {\rm Nullity}(\varphi_{k})=5 \\ \nonumber
k=9 &\quad \quad &{\rm Index}(\varphi_{k})=297 &\quad \quad& {\rm Nullity}(\varphi_{k})=5 \\ \nonumber
k=10 &\quad \quad &{\rm Index}(\varphi_{k})=365&\quad \quad& {\rm Nullity}(\varphi_{k})=5 \\ \nonumber
k=17 &\quad \quad &{\rm Index}(\varphi_{k})=1065 &\quad \quad& {\rm Nullity}(\varphi_{k})=5 \\ \nonumber
k=155 &\quad \quad &{\rm Index}(\varphi_{k})=88433 &\quad \quad& {\rm Nullity}(\varphi_{k})=5  \nonumber
\end{array}
\end{equation}

\textbf{Conjecture:} The ${\rm Nullity}(\varphi_{k})$ is equal to $5$ for any $k$. \vspace{2mm}

We have checked this conjecture by means of a computer algorithm for all $k \leq 1500$. Therefore, it is reasonable to believe that ${\rm Nullity}(\varphi_{k})=5$ for all $k \in \n^*$. The difficulty to prove this conjecture is the following: there are values which are very close both to satisfy $\lambda_{m,n}^-=0$ and to be integers. For instance, the expression which defines $\lambda_{m,n}^-$ vanishes when $k=192,\,m=100$ and $n \simeq 184,998$.

\end{remark}
\end{example}

\begin{example}\label{parallel-r-harmonic-circles}
Let $\varphi_{k}\,: \s^1 \to \s^2\hookrightarrow  \R^3$ be the proper biharmonic map defined by
\begin{equation}\label{r-harmonic-examples}
 \gamma \mapsto \, \left ( \frac{1}{\sqrt 2}\, \cos(k\gamma),\,\frac{1}{\sqrt 2}\,\sin(k\gamma), \,\frac{1}{\sqrt 2}\right ) \,, \quad 0 \leq \gamma \leq 2\pi\,,
\end{equation}
where $k \in \n^*$ is a fixed positive integer. Both the notions of biharmonicity and that of index and nullity of a biharmonic map are invariant under homothetic changes of the metric of either the domain or the codomain. Therefore, in this example, we have assumed for simplicity that the domain is the unit circle. In particular, the radius of the domain which would ensure the condition of isometric immersion for $k=1$ is $R=1/\sqrt 2$, but any choice of $R$ would not affect the conclusions of our next result:
\begin{theorem}\label{Index-theorem-r-harmonic-circles} Let $\varphi_{k} : \s^1 \to \s^2$ be the proper biharmonic map defined in \eqref{r-harmonic-examples}. Then
\begin{align}\label{*}
& {\rm Nullity}(\varphi_{k})=3 \\
& {\rm Index}(\varphi_{k})=1+2(k-1) \nonumber
\end{align}
\end{theorem}
\begin{remark}\label{remark-inex-depend-k} Theorem\link\ref{Index-theorem-r-harmonic-circles} was known in the case that $k=1$: it was proved in \cite{Balmus,LO}, where the index and the nullity of $i:\s^{n-1}(1/\sqrt 2) \hookrightarrow \s^n$ was computed although the pull-back bundle $i^{-1}T\s^n$ is not parallelizable for all dimensions $n$. Since the $k$-fold rotation $e^{{\mathrm i}\vartheta} \mapsto e^{{\mathrm i}k\vartheta}$ of $\s^1$ is a local homothety, it is somehow surprising that the index depends on $k$. A possible explanation is the fact that the $k$-fold rotation
is not a global diffeomorphism and, while the biharmonic equation can be solved locally (so biharmonicity remains invariant under the composition with a local homothety), the index is a global notion.
\end{remark}
\end{example}
\begin{example}\label{example-Legendre} Now we study an example following the lines of \cite{BFO, S1}.  Let
$\mathbb{S}^{2n+1}=\{z\in\mathbb{C}^{n+1}: |z|=1\}$ be the
unit $(2n+1)$-dimensional Euclidean sphere. Consider
$\mathcal{J}:\mathbb{C}^{n+1}\to\mathbb{C}^{n+1}$,
$\mathcal{J}(z)={\mathrm i}z$, to be the usual complex structure on
$\mathbb{C}^{n+1}$ and
$$
\phi=s\circ\mathcal{J},\qquad \xi_z=-\mathcal{J}z,
$$
where $s:T_z\mathbb{C}^{n+1}\to T_z\mathbb{S}^{2n+1}$ is the
orthogonal projection. Endowed with these tensors and the standard
metric $h$, the sphere $(\mathbb{S}^{2n+1},\phi,\xi,\eta,h)$
becomes a Sasakian space form with constant $\varphi$-sectional
curvature equal to $1$.  An isometric immersion $\varphi:M^m\to\mathbb{S}^{2m+1}$ is said to be {\it Legendre} if it is {\it integral}, that is  $\eta(d\varphi(X))=0$ for all $X\in {\mathcal C}(TM)$. 
Sasahara studied the proper biharmonic Legendre
immersed surfaces in Sasakian space forms and obtained the
explicit representations of such surfaces into $\s^5$. In particular, he proved

\begin{theorem}[\cite{S1}]\label{th:Chen_T^2 in S^5}
Let $\varphi:M^2\to\mathbb{S}^5$ be a proper biharmonic Legendre
immersion. Then $\Phi=i\circ\varphi:M^2\to \C^3=\mathbb{R}^6$ is
locally given by
$$
\Phi(\gamma,\vartheta)=\frac{1}{\sqrt 2}\left (e^{{\mathrm i}\gamma},{\mathrm i}e^{-{\mathrm i}\gamma}\sin(\sqrt 2 \vartheta),
{\mathrm i}e^{-{\mathrm i}\gamma}\cos(\sqrt 2 \vartheta)\right),
$$
where $i:\mathbb{S}^5 \hookrightarrow\mathbb{R}^6$ is the canonical
inclusion.
\end{theorem}

The map $\varphi$ induces a full proper biharmonic Legendre embedding
of the flat torus ${\mathbb T}^2=\mathbb{S}^1\times\mathbb{S}^1(1/\sqrt 2)$
into $\mathbb{S}^5$. This embedding, still denoted by $\varphi$, has constant mean curvature $|H|=1/2$, it
is not pseudo-umbilical and its mean curvature vector field is not
parallel. Moreover, $\Phi=\Phi_p+\Phi_{q}$, where
$$
\Phi_{p}(\gamma,\vartheta)=\frac{1}{\sqrt 2}(e^{{\mathrm i}\gamma},0,0),
$$
$$
\Phi_{q}(\gamma,\vartheta)=\frac{1}{\sqrt 2}\big(0,{\mathrm i}e^{-{\mathrm i}\gamma}\sin(\sqrt 2 \vartheta),
{\mathrm i}e^{-{\mathrm i}\gamma}\cos(\sqrt 2 \vartheta) \big),
$$
and $ \Delta \Phi_{p}=\Phi_{p}$, $ \Delta \Phi_{q}=3\Phi_{q}$.
Thus $\Phi$ is a $2$-type mass-symmetric immersion in
$\mathbb{R}^6$ with eigenvalues $1$ and $3$ and order $[1,3]$ (see \cite{BFO, S1} for more details).
Our goal is to determine the index and the nullity of the above
embedding. We shall prove:
\begin{theorem}\label{th: Sasahara_T2_S5}
Let $\varphi:{\mathbb T}^2=\mathbb{S}^1\times\mathbb{S}^1(1/\sqrt
2)\to\mathbb{S}^5$ be the proper biharmonic Legendre embedding. Then
\[
\Index(\varphi)= 11 \quad {\rm and} \quad \nul(\varphi)= 18 \,.
\]
\end{theorem}
\begin{remark} Our result completes the analysis of \cite{BFO}, where it was shown that $\Index(\varphi)\geq 11$ and $\nul(\varphi)\geq 18$.
\end{remark}
\begin{remark} The index and nullity of the biharmonic immersions into spheres that derive from the minimal generalised Veronese immersions have been estimated in \cite{LO}. These maps are pseudo-umbilical immersions with parallel mean curvature vector field and do not have parallelizable pull-back bundle.
\end{remark}
\end{example}
\begin{example}\label{example-noncompact} In this example we study  the notion of stability when the domain is not compact. In this case, the most natural approach is to study the second variation \eqref{Hessian-definition} assuming that $\varphi_{t,s}  = \varphi$ outside a compact set. Variations of this type are called \textit{compactly supported variations}. Using this type of variations, we study the Hessian bilinear form
\begin{equation}\label{hessian-noncompact}
H(E_2)_{(\varphi; D)} (V,W)= \left . \frac{\partial^2}{\partial t \partial s}
\right |_{(t,s)=(0,0)} \hspace{-5mm} E_2 (\varphi_{t,s};D)= \int_D \langle I_2(V),W \rangle  dv_M=\int_M \langle I_2(V),W \rangle  dv_M\,,
\end{equation}
where $D$ is a compact set (with smooth boundary), $\varphi_{t,s}=\varphi$ outside $D$ and the vector fields $V,W$ are defined precisely as in \eqref{V-W} and so they are sections of $\varphi^{-1}TN$ which vanish outside $D$. 
In particular, if $N=\s^n$, then the explicit expression of $I_2(V)$ can be computed again by using the divergence theorem, now for compactly supported vector fields, and we get the same formula \eqref{I2-general-case}. 
The spectrum of $I_2$ is not discrete in general, but we can say that a biharmonic map $\varphi$ is \textit{strictly stable} if
$H(E_2)_{(\varphi; D)} (V,V)>0$ for all nontrivial compactly supported vector fields $V \in \mathcal{C}\left (\varphi^{-1}TN\right)$.
\begin{remark}
It can be proved that  when $M$ is not compact and $N$ is flat, then any proper biharmonic map $\varphi: M\to N$ is stable, that is $H(E_2)_{(\varphi; D)} (V,V)\geq 0$ for all  compactly supported vector fields $V \in \mathcal{C}\left (\varphi^{-1}TN\right)$.
\end{remark}

Now we can describe our example. Let $\varphi\,: \R \to \s^2$ be the proper biharmonic map defined by
\begin{equation}\label{A-harmonic-examples}
 \gamma \mapsto \, \left ( \cos(A(\gamma)),\,\sin(A(\gamma)), \,0\right ) \in \s^2\hookrightarrow  \R^3 \,,
\end{equation}
with
\begin{equation}\label{A-gamma}
A(\gamma)= a \gamma^3+b \gamma^2+c \gamma +d \,,
\end{equation}
where $a,b,c,d$ are real numbers such that $a^2+b^2>0$. We shall prove:
\begin{theorem}\label{th:noncompact} Let $\varphi\,: \R \to \s^2$ be the proper biharmonic map defined in \eqref{A-harmonic-examples}. Assume that
\begin{equation}\label{condiz-noncompact}
{\rm either} \,\, a=0 \quad {\rm or} \,\, \left\{a\neq 0 \,\,{\rm and} \,\, b^2 -3ac \leq 0 \right \}\,.
\end{equation}
Then $\varphi$ is strictly stable.
\end{theorem}
We point out that this is the first example of a strictly stable, proper biharmonic map into a sphere. In general, according to a result of Jiang \cite{Jiang}, we know that, when $M$ is compact and the tension field $\tau(\varphi)$ is orthogonal to the image, then any proper biharmonic map $\varphi:M\to \s^n$ is unstable. By contrast, in the example of Theorem~\ref{th:noncompact} the tension field is tangent to the image of the map.
\end{example}
\section{Proof of the results}\label{proofs}
We shall first prove Theorems\link\ref{Spectrum-theorem-toro-sfera} and \ref{Index-theorem-toro-sfera}. The first step is to derive an explicit formula for the operator $I_2:\mathcal{C}\left(\varphi_k^{-1} T\s^2\right) \to \mathcal{C}\left(\varphi_k^{-1} T\s^2\right)$ using its explicit expression \eqref{I2-general-case}. To this purpose, it is convenient to introduce two suitable vector fields along $\varphi_k$. More specifically, using coordinates $\left(y^1,y^2,y^3 \right)$ on $\R^3$, we define
\begin{equation}\label{definizioneY-eta}
Y=\sqrt 2 \,\left (-y^2,y^1,0 \right ) \quad {\rm and} \quad \eta=
\left(y^1,y^2,\,-\,\frac{1}{\sqrt 2} \right)\,.
\end{equation}
From a geometric viewpoint, we observe that the image of $\varphi_k$ is a circle $\s^1 (1 /\sqrt 2)$ into $\s^2$. Then the restriction of $Y$ to the circle provides a unit section of $T\s^1 (1 /\sqrt 2)$, while $\eta$ gives rise to a unit section of the normal bundle of $\s^1 (1 /\sqrt 2)$ into $\s^2$. For our future purposes, we shall use the following elementary calculations:
\begin{equation}\label{IIfund-form}
B(Y,Y)=-\eta \,\, ; \quad A(Y)=-Y \,\,,
\end{equation}
where $B$ and $A$ denote the second fundamental form and the shape operator respectively. Then, we set
\begin{equation}\label{def-V1 e V2}
V_Y=Y \left (\varphi_k \right )\quad {\rm and} \quad V_\eta= \eta \left ( \varphi_k \right )\,.
\end{equation}
The vectors $V_Y$, $V_\eta$ provide an orthonormal basis on $T\s^2$ at each point of the image of $\varphi_k$ and it is easy to conclude that each section $V \in \mathcal{C}\left( \varphi_k^{-1}T\s^2\right )$ can be written as
\begin{equation}\label{general-section-toro-sfera}
V=f_1 \,V_Y +f_2\,V_\eta \,,
\end{equation}
where $f_j\in C^\infty \left ( {\mathbb T}^2 \right)$, $j=1,2$. For our purposes, it shall be sufficient to study in detail the case that the functions $f_j$ are eigenfunctions of the Laplacian. More precisely, let
\[
\Delta= -\left (\frac{\partial^2}{\partial \gamma^2}+ \frac{\partial^2}{\partial \vartheta^2}\right )
\]
be the Laplace operator on ${\mathbb T}^2$ and denote by $\lambda_i,\,i \in \n$, its spectrum. We define
\begin{equation}\label{sottospazi-S-lambda}
 S^{\lambda_i}=\left \{ f_1\,V_Y \,\,:\,\, \Delta f_1= \lambda_i f_1 \right \} \oplus \left \{ f_2\,V_\eta \,\,:\,\, \Delta f_2= \lambda_i f_2 \right \}
 \end{equation}
As in \cite{LO},  $S^{\lambda_i} \perp S^{\lambda_j}$ if $i \neq j$ and $\oplus_{i=0}^{+\infty}\, S^{\lambda_i}$ is dense in
$\mathcal{C}\left( \varphi_k^{-1}T\s^2\right )$ (note that the scalar product which we use on sections of $ \varphi_k^{-1}T\s^2$ is the standard $L^2$-inner product). Our first key result is:
\begin{proposition}\label{proposizione-I-esplicito-toro-sfera} Assume that $f\in C^{\infty}\left ( {\mathbb T}^2 \right )$ is an eigenfunction of $\Delta$ with eigenvalue $\lambda$. Then
\begin{equation}\label{prima-espressione-I-toro-sfera}
I_2 (f V_Y)= \lambda \left ( \lambda+k^2 \right )fV_Y -2 k^2\,f_{\gamma \gamma}V_Y+2 \sqrt 2 k \lambda f_{\gamma}V_\eta
\end{equation}
and
\begin{equation}\label{seconda-espressione-I-toro-sfera}
I_2 (f V_\eta)= \left ( \lambda^2-k^4 \right )fV_\eta -2 k^2 f_{\gamma \gamma}V_\eta-2 \sqrt 2 k \lambda f_{\gamma} V_Y \,.
\end{equation}
\end{proposition}
\begin{proof} We recall that the definition of the map $\varphi_k$ was given in \eqref{equivdatoroasfera}. The vector fields $\partial / \partial \gamma$ and $\partial / \partial \vartheta$ quotient to vector fields tangent to $\t^2$ forming a global orthonormal frame field of ${\mathbb T}^2$ and we easily find:
\begin{equation}\label{dfi}
d\varphi_k \left (\frac{\partial}{\partial \gamma} \right )= k \,\frac{\sqrt 2}{2}\, Y \left ( \varphi_k \right )=
k \,\frac{\sqrt 2}{2}\, V_Y\quad {\rm and} \quad
d\varphi_k \left (\frac{\partial}{\partial \vartheta} \right )= 0 \,,
\end{equation}
where we have used the vector fields introduced in \eqref{definizioneY-eta} and \eqref{def-V1 e V2}. In order to complete the proof of Proposition\link\ref{proposizione-I-esplicito-toro-sfera} we need to compute all the terms which appear in formula \eqref{I2-general-case}. This shall be done by means of a series of lemmata (to simplify notation, in these lemmata we shall write $\varphi$ instead of $\varphi_k$).
\begin{lemma}\label{lemma1}
\begin{equation}\label{rough-V1}
\overline{\Delta} V_Y=\frac{1}{2}\,k^2 \,V_Y \,.
\end{equation}
\end{lemma}
\begin{proof}[Proof of Lemma~\ref{lemma1}] In general, for $X \in \mathcal{C}(T\,{\mathbb T}^2)$ we have:
\begin{eqnarray}\label{1-lemma-1}
\nabla_X^\varphi V_Y&=&\nabla_X^\varphi Y(\varphi)=\nabla_{d\varphi(X)}^{\s^2} \,Y\\ \nonumber
&=&\nabla_{d\varphi(X)}^{\s^1(1/\sqrt 2)}\, \,Y +B(d\varphi(X),Y)\,. \nonumber
\end{eqnarray}
If we apply \eqref{1-lemma-1} to $X=\partial / \partial \gamma$ we easily obtain
\begin{equation}\label{2-lemma-1}
\nabla_{\partial / \partial \gamma}^\varphi \,V_Y = \frac{\sqrt 2}{2}\,k\,B(Y,Y)=-\,\frac{\sqrt 2}{2}\,k\,V_{\eta}\,,
\end{equation}
where for the last equality we have used \eqref{IIfund-form}. Next,
\begin{eqnarray}\label{3-lemma-1}
\nabla_{\partial / \partial \gamma}^\varphi \,\,\left (\nabla_{\partial / \partial \gamma}^\varphi \,V_Y \right )&=&-\,\frac{\sqrt 2}{2}\,k\,
\nabla_{\partial / \partial \gamma}^\varphi \,V_{\eta}\\ \nonumber
&=&-\,\frac{\sqrt 2}{2}\,k\,
\nabla_{(\sqrt 2/2)kY}^{\s^2} \,\eta\\ \nonumber
&=&-\,\frac{1}{2}\,k^2\,(-A(Y)+0)\\ \nonumber
&=& -\,\frac{1}{2}\,k^2\,V_Y \,, \nonumber
\end{eqnarray}
where for the last equality we have used \eqref{IIfund-form}. Since $d\varphi (\partial / \partial \vartheta)$ vanishes, \eqref{rough-V1} follows immediately from \eqref{3-lemma-1} (note the sign convention).
\end{proof}
\begin{lemma}\label{lemma2}
\begin{equation}\label{rough-V2}
\overline{\Delta} V_\eta=\frac{1}{2}\,k^2 \,V_\eta \,.
\end{equation}
\end{lemma}
\begin{proof}[Proof of Lemma~\ref{lemma2}]
\begin{equation}\label{1-lemma-2}
\nabla_{\partial / \partial \gamma}^\varphi \,V_\eta=
\nabla_{(\sqrt 2/2)kY}^{\s^2} \,\eta= \frac{\sqrt 2}{2}k V_Y \,,
\end{equation}
from which
\begin{equation}\label{2-lemma-2}
\nabla_{\partial / \partial \gamma}^\varphi
\left (\nabla_{\partial / \partial \gamma}^\varphi V_\eta \right )=
\frac{1}{2}\,k^2 \nabla_Y^{\s^2}Y= -\,\frac{1}{2} k^2\,V_\eta \,.
\end{equation}
Now \eqref{rough-V2} follows readily.
\end{proof}
\begin{lemma}\label{lemma3} Assume that $\Delta f= \lambda f$. Then
\begin{eqnarray}\label{rough-fV1-fV2}
{\rm (i)}\,\,\quad \overline{\Delta} (fV_Y)&=&\left ( \lambda+\frac{k^2}{2}\right )fV_Y+ \sqrt 2kf_\gamma V_\eta\\ \nonumber
{\rm (ii)}\,\, \quad \overline{\Delta} (f V_\eta)&=&
\left ( \lambda+\frac{k^2}{2}\right )f V_\eta- \sqrt 2 k f_\gamma V_Y\nonumber
\end{eqnarray}
\end{lemma}
\begin{proof}[Proof of Lemma~\ref{lemma3}]
This lemma can be easily proved by applying the results of Lemmata\link\ref{lemma1} and \ref{lemma2} to the general formula
\begin{equation}\label{general-product-formula}
\overline{\Delta} (f\,V)=(\Delta f)\,V-2\,\nabla_{\nabla f}^{\varphi}V+f\,\overline{\Delta} V \,.
\end{equation}
\end{proof}
Now, we compute the various terms which appear in the formula \eqref{I2-general-case}.
\begin{lemma}\label{lemma4}Assume that $\Delta f= \lambda f$. Then
\begin{equation}\label{doppio-rough-fV1}
\overline{\Delta}^2 (f V_Y)=\left ( \lambda+\frac{k^2}{2}\right )^2 f V_Y-2k^2 f_{\gamma \gamma} V_Y+ 2\sqrt 2 k \left ( \lambda+\frac{k^2}{2}\right )f_\gamma V_\eta \,.
\end{equation}
\end{lemma}
\begin{proof}[Proof of Lemma~\ref{lemma4}] We just need to compute using twice \eqref{rough-fV1-fV2}\link(i) together with the observation that, since $\partial / \partial \gamma$ is a Killing field on ${\mathbb T}^2$, $\Delta f_\gamma= \lambda\,f_\gamma$.
\end{proof}
In the same way, we obtain:
\begin{lemma}\label{lemma5}Assume that $\Delta f= \lambda f$. Then
\begin{equation}\label{doppio-rough-fV2}
\overline{\Delta}^2 (f V_\eta)=\left ( \lambda+\frac{k^2}{2}\right )^2 f V_\eta-2k^2 f_{\gamma \gamma} V_\eta- 2\sqrt 2 k \left ( \lambda+\frac{k^2}{2}\right )f_\gamma V_Y \,.
\end{equation}
\end{lemma}
\begin{lemma}\label{lemma6}
\begin{equation}\label{1-lemma6}
\overline{\Delta}\left ( {\rm trace}\langle f V_Y,d\varphi \cdot \rangle d\varphi \cdot - |d\varphi|^2\,fV_Y \right)=0 \,.
\end{equation}
\end{lemma}
\begin{proof}[Proof of Lemma~\ref{lemma6}]
\[
{\rm trace}\langle fV_Y,d\varphi \cdot \rangle d\varphi \cdot - |d\varphi|^2 fV_Y=\langle fY,\frac{\sqrt 2}{2}kY\rangle  \frac{\sqrt 2}{2}kY -\frac{k^2}{2} fY =0
\]
\end{proof}
\begin{lemma}\label{lemma7}
\begin{equation}\label{1-lemma7}
\overline{\Delta}\left ( {\rm trace}\langle fV_\eta,d\varphi \cdot \rangle d\varphi \cdot - |d\varphi|^2\,fV_\eta \right)=
-\,\frac{k^2}{2}\left ( \lambda+\frac{k^2}{2}\right )f\,V_\eta+ \frac{k^3 \sqrt 2}{2}\,f_\gamma\,V_Y\,.
\end{equation}
\end{lemma}
\begin{proof}[Proof of Lemma~\ref{lemma7}]
\[
{\rm trace}\langle fV_\eta,d\varphi \cdot \rangle d\varphi \cdot - |d\varphi|^2\,fV_\eta=\langle f\eta,\frac{\sqrt 2}{2}kY\rangle \,\frac{\sqrt 2}{2}k V_Y -\frac{k^2}{2} fV_\eta = -\frac{k^2}{2} fV_\eta \,.
\]
Next, using \eqref{rough-fV1-fV2}\link(ii), we obtain \eqref{1-lemma7}.
\end{proof}
\begin{lemma}\label{lemma8}
\begin{equation}\label{1-lemma8}
2\langle d\tau(\varphi),d\varphi \rangle fV_Y+|\tau(\varphi)|^2 f
V_Y=-\,\frac{k^4}{4}\,f\,V_Y
\end{equation}
\end{lemma}
\begin{proof}[Proof of Lemma~\ref{lemma8}]The tension field of an equivariant map of the type \eqref{equivdatoroasfera} is
\[
\tau(\varphi)=\left ( \alpha''- \frac{k^2}{2}\sin (2 \alpha) \right ) \frac{\partial}{\partial \alpha} \,.
\]
Since $\alpha \equiv \pi /4$ we can write, after standard identification, 
\begin{equation}\label{tau}
\tau(\varphi)=- \frac{k^2}{2}\,V_\eta \,.
\end{equation}
Now
\begin{eqnarray}\label{2-lemma8}
2\langle d\tau(\varphi),d\varphi \rangle fV_Y&=&2 \langle\nabla_
{\partial/\partial \gamma}^\varphi \left ( - \frac{k^2}{2} V_\eta\right ), \frac{\sqrt 2}{2}k V_Y \rangle fV_Y \\ \nonumber
&=&-\,\frac{k^4}{2}\,f\,V_Y \,, \nonumber
\end{eqnarray}
from which it is immediate to obtain \eqref{1-lemma8}.
\end{proof}
In a similar way we obtain:
\begin{lemma}\label{lemma9}
\begin{equation}\label{1-lemma9}
2\langle d\tau(\varphi),d\varphi \rangle fV_\eta+|\tau(\varphi)|^2 f
V_\eta=-\,\frac{k^4}{4}\,f\,V_\eta
\end{equation}
\end{lemma}
\begin{lemma}\label{lemma11}
\begin{equation}\label{1-lemma11}
-2\,{\rm trace}\langle fV_Y,d\tau(\varphi) \cdot \rangle d\varphi \cdot=
\frac{k^4}{2}\,f\,V_Y\,.
\end{equation}
\end{lemma}
\begin{proof}[Proof of Lemma~\ref{lemma11}]
\begin{eqnarray}\label{2-lemma11}
-2\,{\rm trace}\langle fV_Y,d\tau(\varphi) \cdot \rangle d\varphi \cdot&=&-2 \langle fV_Y, \nabla_
{\partial/\partial \gamma}^\varphi \left ( - \frac{k^2}{2} V_\eta\right ) \rangle \frac{\sqrt 2}{2}k V_Y \\ \nonumber
&=&\frac{\sqrt 2}{2}k^3 f \langle Y,\frac{\sqrt 2}{2}kY  \rangle  V_Y \\ \nonumber
&=&\frac{k^4}{2}\,f\,V_Y \,,\nonumber
\end{eqnarray}
\end{proof}
\begin{lemma}\label{lemma12}
\begin{equation}\label{1-lemma12}
-2\,{\rm trace}\langle fV_\eta,d\tau(\varphi) \cdot \rangle d\varphi \cdot= 0\,.
\end{equation}
\end{lemma}
\begin{proof}[Proof of Lemma~\ref{lemma12}]
\begin{eqnarray}\label{2-lemma12}
-2\,{\rm trace}\langle fV_\eta,d\tau(\varphi) \cdot \rangle d\varphi \cdot&=&-2 \langle fV_\eta, \nabla_
{\partial/\partial \gamma}^\varphi \left ( - \frac{k^2}{2} V_\eta\right ) \rangle \frac{\sqrt 2}{2}k V_Y \\ \nonumber
&=&\frac{\sqrt 2}{2}k^3 f \langle \eta,\frac{\sqrt 2}{2}kY  \rangle  V_Y =0  \nonumber
\end{eqnarray}
\end{proof}
All the calculations involved in the next two lemmata use the same patterns which we followed so far and so we state directly the relevant results omitting the details of the proofs.
\renewcommand{\arraystretch}{1.3}
\begin{lemma}\label{lemma13}
\[\begin{array}{rcl}
-2 \,{\rm trace} \langle \tau(\varphi),d(fV_Y) \cdot \rangle d\varphi \cdot &=& -\,\frac{k^4}{2}fV_Y\\ \nonumber
- \langle \tau(\varphi),fV_Y \rangle \tau(\varphi)&=&0 \\ \nonumber
{\rm trace} \langle d\varphi \cdot,\overline{\Delta}(fV_Y) \rangle d\varphi \cdot &=&\frac{k^2}{2}\left ( \lambda+\frac{k^2}{2}\right )f V_Y \\ \nonumber
{\rm trace}\langle d\varphi \cdot,\left ( {\rm trace}\langle
fV_Y,d\varphi \cdot \rangle d\varphi \cdot \right ) \rangle d\varphi \cdot&=& \frac{k^4}{4}fV_Y
\end{array}
\]
\[\begin{array}{rcl}
-2 |d\varphi|^2\, {\rm trace}\langle d\varphi \cdot,fV_Y \rangle
 d\varphi \cdot&=&- \frac{k^4}{2}fV_Y \\ \nonumber
 2 \langle d(fV_Y),d\varphi\rangle \tau(\varphi) &=&
 - \frac{k^3 \sqrt 2}{2} f_\gamma V_\eta\\ \nonumber
-|d\varphi|^2 \,\overline{\Delta}(fV_Y)&=& -\,\frac{k^2}{2}\left ( \lambda+\frac{k^2}{2}\right )f\,V_Y- \frac{k^3 \sqrt 2}{2} f_\gamma V_\eta\\ \nonumber
|d\varphi|^4 fV_Y&=&\frac{k^4}{4}fV_Y \nonumber
\end{array}
\]
\end{lemma}
\begin{lemma}\label{lemma14}
\[\begin{array}{rcl}
-2 \,{\rm trace} \langle \tau(\varphi),d(fV_\eta) \cdot \rangle d\varphi \cdot &=& \frac{k^3\sqrt 2}{2}f_\gamma V_Y\\ \nonumber
- \langle \tau(\varphi),fV_\eta \rangle \tau(\varphi)&=&-\,\frac{k^4}{4}fV_\eta \\ \nonumber
{\rm trace} \langle d\varphi \cdot,\overline{\Delta}(fV_\eta) \rangle d\varphi \cdot &=&-\,\frac{k^3\sqrt 2}{2}f_\gamma V_Y \\ \nonumber
{\rm trace}\langle d\varphi \cdot,\left ( {\rm trace}\langle
f V_\eta,d\varphi \cdot \rangle d\varphi \cdot \right ) \rangle d\varphi \cdot&=& 0
\end{array}
\]
\[\begin{array}{rcl}
-2 |d\varphi|^2\, {\rm trace}\langle d\varphi \cdot,fV_\eta \rangle
 d\varphi \cdot&=&0 \\ \nonumber
 2 \langle d(fV_\eta),d\varphi\rangle \tau(\varphi) &=&
 - \frac{k^4}{2}\,f V_\eta\\ \nonumber
-|d\varphi|^2 \,\overline{\Delta}(fV_\eta)&=& -\,\frac{k^2}{2}\left ( \lambda+\frac{k^2}{2}\right )f\,V_\eta+ \frac{k^3 \sqrt 2}{2}\,f_\gamma\,V_Y\\ \nonumber
|d\varphi|^4 fV_\eta&=&\frac{k^4}{4}fV_\eta \,\,.\nonumber
\end{array}
\]
\end{lemma}
\renewcommand{\arraystretch}{1.}
Now we are able to end the proof of Proposition\link\ref{proposizione-I-esplicito-toro-sfera}. As for \eqref{prima-espressione-I-toro-sfera}, it suffices to replace the results of Lemmata\link\ref{lemma4}, \ref{lemma6}, \ref{lemma8}, \ref{lemma11} and \ref{lemma13} into \eqref{I2-general-case} and add up. Similarly, \eqref{seconda-espressione-I-toro-sfera} can be obtained using Lemmata\link\ref{lemma5}, \ref{lemma7}, \ref{lemma9}, \ref{lemma12} and \ref{lemma14}.
\end{proof}
We are now in the right position to prove our main theorems.
\subsection{Proof of Theorem~\ref{Spectrum-theorem-toro-sfera}}
The eigenvalues of $\Delta$ on ${\mathbb T}^2$ have the form $\lambda=m^2+n^2$. In particular, $\lambda_0=0$,
\begin{equation}\label{sottospazi-S-lambda-0}
 S^{\lambda_0}=\left \{ c_1\,V_Y \,\,:\,\, c_1 \in \R\right \} \oplus \left \{ c_2\,V_\eta \,\,:\,\, c_2\in \R\right \}
 \end{equation}
and $\dim \left( S^{\lambda_0} \right)=2$. It follows by a direct application of Proposition\link\ref{proposizione-I-esplicito-toro-sfera} that the restriction of $I_2$ to $S^{\lambda_0}$ gives rise to the eigenvalues $\mu_0=0$ and $\mu_1=-k^4$.
Next, let us consider the case that $\lambda=m^2+n^2 >0$ and denote by $W_{\lambda}$ the corresponding eigenspace. In a similar fashion to \cite{BFO}, we decompose
\begin{equation}\label{sottospazi-W-lambda}
W_{\lambda}=W^{m,0}\oplus_{m,n \geq 1}W^{m,n}\oplus W^{0,n}\,,
 \end{equation}
where it is understood that in \eqref{sottospazi-W-lambda} we have to consider all the possible couples $(m,n)\in \n \times \n$ such that $\lambda=m^2+n^2$. By way of example, if $\lambda=4$, then the possible couples are $(2,0)$ and $(0,2)$. If $\lambda=5$, then the possible couples are $(1,2)$ and $(2,1)$. The subspaces of the type $W^{m,0}$ are $2$-dimensional and are spanned by the functions $\left \{\cos (m\gamma),\sin (m\gamma) \right \}$. Similarly, $W^{0,n}$ is $2$-dimensional and is generated by $\left \{\cos (n\vartheta),\sin (n\vartheta) \right \}$. Finally, the subspaces $W^{m,n}$, with $m,n \geq 1$, have dimension $4$ and are spanned by
\[
\left \{\cos (m\gamma)\cos(n\vartheta),\cos (m\gamma)\sin (n\vartheta),\sin (m\gamma)\cos(n\vartheta),\sin (m\gamma)\sin (n\vartheta)\right \}
\]
Now it becomes natural to define
\begin{equation}\label{sottospazi-S-m,n}
 S^{m,n}=\left \{ f_1\,V_Y \,\,:\,\, f_1 \in W^{m,n} \right \} \oplus \left \{ f_2\,V_\eta \,\,:\,\, f_2 \in W^{m,n} \right \} \,\,.
 \end{equation}
All these subspaces are orthogonal to each other. Moreover, for any positive eigenvalue $\lambda_i$, we have
\begin{equation}\label{decomposizioneS-lambda-in-coppie-m-n}
 S^{\lambda_i}= \oplus_{m^2+n^2=\lambda_i}\,\,S^{m,n}\,.
 \end{equation}
It follows easily from Proposition\link\ref{proposizione-I-esplicito-toro-sfera} that the operator $I_2$ preserves each of the subspaces $S^{m,n}$. Therefore, its spectrum can be computed by determing the eigenvalues of the matrices associated to the restriction of $I_2$ to each of the $S^{m,n}\,$'s. We separate three cases:

\textbf{Case 1:} $S^{m,0}$, $m \geq1$.

In this case, an orthonormal basis of $S^{m,0}$ is given by:
\[
\left \{\frac{\cos (m\gamma)}{\sqrt{2} \,\pi}\,V_Y,\frac{\sin (m\gamma)}{\sqrt{2} \,\pi}\,V_Y,\frac{\cos (m\gamma)}{\sqrt{2} \,\pi}\,V_\eta,\frac{\sin (m\gamma)}{\sqrt{2} \,\pi}\,V_\eta\right \} \,.
\]
Using Proposition\link\ref{proposizione-I-esplicito-toro-sfera} and computing we find that in this case the $(4\times 4)$-matrices associated to the operator $I_2$ are:
\begin{equation}\label{matrici-S-m-0}
\left(
\begin{array}{cccc}
 m^2\left(m^2+3k^2 \right ) & 0 & 0 &  -2\sqrt 2\,k m^3 \\
 0&  m^2\left(m^2+3k^2 \right )&2 \sqrt 2\,k m^3&0\\
 0&2 \sqrt 2\,k m^3 & m^4+2k^2m^2-k^4 &0 \\
 -2 \sqrt 2\,k m^3 & 0&0&m^4+2k^2m^2-k^4  \\
\end{array}
\right)
\end{equation}
The eigenvalues of these matrices are precisely the $\lambda_{m,0}^+\,$'s, $\lambda_{m,0}^-\,$'s indicated in the statement of Theorem\link\ref{Spectrum-theorem-toro-sfera}. Each of them has multiplicity equal to $2$.

\textbf{Case 2:} $S^{0,n}$, $n \geq1$.

In this case, an orthonormal basis of $S^{0,n}$ is given by:
\[
\left \{ \frac{\cos (n\vartheta)}{\sqrt{2} \,\pi}\,V_Y,\frac{\sin (n\vartheta)}{\sqrt{2} \,\pi}\,V_Y,\frac{\cos (n\vartheta)}{\sqrt{2} \,\pi}\,V_\eta,\frac{\sin (n\vartheta)}{\sqrt{2} \,\pi} \,V_\eta\right \} \,.
\]
Using Proposition\link\ref{proposizione-I-esplicito-toro-sfera} and computing it is immediate to find that in this case the $(4\times 4)$-matrices associated to the operator $I_2$ are:
\begin{equation}\label{matrici-S-0-n}
\left(
\begin{array}{cccc}
 n^2\left(n^2+k^2 \right ) & 0 & 0 & 0 \\
 0&  n^2\left(n^2+k^2 \right )&0&0\\
 0&0 & n^4-k^4 &0 \\
 0 & 0&0&n^4-k^4  \\
\end{array}
\right)
\end{equation}
The eigenvalues of these matrices are obviously those indicated with $\lambda_{0,n}^+,\lambda_{0,n}^-$ in the statement of Theorem\link\ref{Spectrum-theorem-toro-sfera}. Each of them has multiplicity equal to $2$.

\textbf{Case 3:} $S^{m,n}$, $m,n \geq1$.

This is the case which requires the biggest computational effort.
An orthonormal basis of $S^{m,n}$ is given by:
\renewcommand{\arraystretch}{1.5}
\[\begin{array}{l}
\left \{ \frac{1}{\pi} \,\cos (m \gamma)\cos (n\vartheta)\,V_Y,
\frac{1}{\pi} \,\cos (m \gamma)\sin (n\vartheta)\,V_Y,
\frac{1}{\pi} \,\sin (m \gamma)\cos (n\vartheta)\,V_Y,
\frac{1}{\pi} \,\sin (m \gamma)\sin (n\vartheta)\,V_Y, \right . \\ \nonumber
\left . \frac{1}{\pi} \,\cos (m \gamma)\cos (n\vartheta)\,V_\eta,
\frac{1}{\pi} \,\cos (m \gamma)\sin (n\vartheta)\,V_\eta,
\frac{1}{\pi} \,\sin (m \gamma)\cos (n\vartheta)\,V_\eta,
\frac{1}{\pi} \,\sin (m \gamma)\sin (n\vartheta)\,V_\eta\right \} \,.
\end{array}
\]
\renewcommand{\arraystretch}{1.}
Using Proposition\link\ref{proposizione-I-esplicito-toro-sfera} and computing we find that in this case the $(8\times 8)$-matrices associated to the operator $I_2$ can be described as follows. Set
\begin{eqnarray*}
A_{m,n,k}&=& \left(3 m^2+n^2\right) k^2+\left(m^2+n^2\right)^2 \\
B_{m,n,k}&= & -k^4+2 m^2 k^2+\left(m^2+n^2\right)^2\\
   C_{m,n,k}&=& 2 \sqrt{2} k m \left(m^2+n^2\right)
\end{eqnarray*}
Then the matrices are:
\begin{equation}\label{matrici-S-m-n}
\left(
\begin{array}{cccccccc}
 A_{m,n,k} & 0 & 0 & 0 & 0 & 0
   & -C_{m,n,k} & 0 \\
 0 & A_{m,n,k} & 0 & 0 & 0 & 0
   & 0 & -C_{m,n,k} \\
 0 & 0 & A_{m,n,k} & 0 & C_{m,n,k} & 0 & 0 & 0 \\
 0 & 0 & 0 & A_{m,n,k}& 0 & C_{m,n,k} & 0 & 0 \\
 0 & 0 & C_{m,n,k} & 0 & B_{m,n,k} & 0 & 0 & 0 \\
 0 & 0 & 0 & C_{m,n,k} & 0 & B_{m,n,k} & 0 & 0 \\
 -C_{m,n,k} & 0 & 0 & 0 & 0 & 0 &B_{m,n,k} & 0 \\
 0 & -C_{m,n,k}& 0 & 0 & 0 & 0 & 0 & B_{m,n,k}
\end{array}
\right)
\end{equation}
The characteristic polynomial of this matrix is:
\[
\left[A_{m,n,k} (\lambda-B_{m,n,k})+\lambda (B_{m,n,k}-\lambda)+C_{m,n,k}^2\right]^4 \,.
\]
Then a straightforward computation shows that the eigenvalues are the $\lambda_{m,n}^{\pm}\,$'s given in the statement of Theorem\link\ref{Spectrum-theorem-toro-sfera}. Each of them has multiplicity equal to $4$ and this ends the proof.
\subsection{Proof of Theorem~\ref{Index-theorem-toro-sfera}}
First, it is obvious that the subspace $S^{\lambda_0}$ yields a contribution of $+1$ for both ${\rm Index}\left( \varphi_k \right)$ and ${\rm Nullity}\left( \varphi_k \right)$.
Next, we examine the subspaces of the type $S^{0,n}$. It is immediate to conclude that their contribution to ${\rm Index}\left( \varphi_k \right)$ is $2(k-1)$, because we have $k-1$ negative eigenvalues of multiplicity $2$. The contribution of $S^{0,n}$ to ${\rm Nullity}\left( \varphi_k \right)$ is always equal to $2$ because $\lambda^{-}_{0,n}$ vanishes if and only if $n=k$, and $\lambda^{+}_{0,n}$ is always positive. The same conclusions hold for the subspaces of the type $S^{m,0}$.
Indeed, it is obvious that $\lambda_{m,0}^+$ is positive for all $m \geq 1$. As for $\lambda_{m,0}^-$, our claim is an immediate consequence of the following lemma:
\begin{lemma}\label{lemma-tecnico1}
If $1\leq m \leq k-1$, then $\lambda_{m,0}^- <0$. If $m=k$, then $\lambda_{m,0}^- =0$. If $m>k$, then $\lambda_{m,0}^- >0$.
\end{lemma}
\begin{proof} [Proof of Lemma~\ref{lemma-tecnico1}]
The eigenvalue $\lambda_{m,0}^- $ has the same sign of the expression
\begin{equation}\label{1}
-k^4+5k^2 m^2+2 m^4- \sqrt{k^8+2 k^6 m^2+k^4
   m^4+32k^2 m^6 }
\end{equation}
If we set $m=ck$ into \eqref{1} we obtain:
\[
k^4 \, \left ( -1+5c^2+2 c^4- \sqrt{1+2 c^2+c^4+32c^6 } \right ) = k^4 \, h(c)\,.
\]
Now a routine analysis shows that $h'(c)$ is positive for $c>0$ and $h(1)=0$, and from this the conclusion of the lemma follows immediately.
\end{proof}
By way of summary, the total contribution of the subspaces $S^{\lambda_0}$, $S^{0,n}$ and $S^{m,0}$ to the index and the nullity of $\varphi_k$ is $1+4(k-1)$ and $5$ respectively. To these values, we have to add the contributions coming from the subspaces of the type $S^{m,n}$ with $m,n \geq 1$. Now, we observe that all the eigenvalues of the type $\lambda_{m,n}^+$ are positive: this follows from the fact that the expression
{\small
\[
-k^4+k^2 \left(5 m^2+n^2\right)+2 \left(m^2+n^2\right)^2 + \sqrt{k^8+2 k^6 \left(m^2+n^2\right)+k^4
   \left(m^2+n^2\right)^2+32 k^2 m^2 \left(m^2+n^2\right)^2}
\]}
vanishes when $m=n=0$ and it is increasing with respect to both $m$ and $n$. 

Therefore, the conclusion of the proof is an immediate consequence of the definition of the functions $f$ and $g$ in \eqref{definizione-f(k)} and \eqref{definizione-g(k)} respectively, together with the fact that each of the $\lambda_{m,n}^-\,$'s has multiplicity $4$.
\begin{remark}
Taking into account the expression of the eigenvectors corresponding to $\mu_0=\lambda^-_{k,0}=\lambda^-_{0,k}=0$ we have the following geometric description of the space where $I_2$ vanishes:
$$
\Big\{d\varphi_k(X)\colon X\; {\rm Killing}\Big\}\oplus W^{0,k} V_{\eta}\oplus \Big\{ \frac{1}{k^2}d\varphi_k(\nabla f)+f V_{\eta}\colon f\in W^{k,0}\Big\}\,.
$$
In a similar way, we can describe the space where $H(E_2)_{\varphi_k}$ is negative definite. For example, if $k=1$, the space where $H(E_2)_{\varphi_1}$ is negative definite is $\left \{ c_2 V_\eta \colon c_2 \in \R \right \}$. If $k=2$ the situation becomes more interesting and here we report the explicit description. More precisely, a computation shows that the space where $H(E_2)_{\varphi_2}$ is negative definite is
\[
\left \{ c_2 V_\eta \colon c_2 \in \R \right \} \oplus W^{0,1}V_\eta \oplus \mathcal{A}_{1,0}\oplus \mathcal{A}_{1,1}\oplus \mathcal{A}_{2,1} \,,
\]
where
\begin{eqnarray*}
\mathcal{A}_{1,0}&=&\left \{ \frac{1}{4}(-5+\sqrt{33})d\varphi_2(\nabla f)+fV_\eta \,\,\colon f \in W^{1,0}\right \}\, ;\\
\mathcal{A}_{1,1}&=&\left \{ \frac{1}{4}(-3+\sqrt{17})d\varphi_2(\nabla f)+fV_\eta \,\,\colon f \in W^{1,1}\right \}\, ;\\
\mathcal{A}_{2,1}&=&\left \{ \frac{1231+41\sqrt{881}}{80\,(59+2\sqrt{881})}\,\,d\varphi_2(\nabla f)+fV_\eta \,\,\colon f \in W^{2,1}\right \} \,.
\end{eqnarray*}

Similar computations can also be performed in the cases $k \geq3$, but they are rather long and so we omit further details.
\end{remark}
\subsection{Proof of Theorem\link\ref{Index-theorem-r-harmonic-circles}}
Again, we use vector fields $V_Y$ and $V_\eta$ defined precisely as in \eqref{def-V1 e V2}.
The vectors $V_Y$, $V_\eta$ provide an orthonormal basis on $T\s^2$ at each point of the image of $\varphi_k$ and it is easy to conclude that each section $V \in \mathcal{C}\left( \varphi_k^{-1}T\s^2\right )$ can be written as
\begin{equation}\label{general-section-toro-sfera-bis}
V=f_1 \,V_Y +f_2\,V_\eta \,,
\end{equation}
where $f_j\in C^\infty \left ( \s^1 \right)$, $j=1,2$. The version of Proposition\link\ref{proposizione-I-esplicito-toro-sfera} in this context is:
\begin{proposition}\label{proposizione-I-esplicito-s1-sfera} Assume that $f\in C^{\infty}\left ( \s^1 \right )$ is an eigenfunction of $\Delta$ with eigenvalue $\lambda$. Then
\begin{equation}\label{prima-espressione-I-s1-sfera}
I_2 (fV_Y)= \lambda \,\left ( \lambda+3k^2 \right )fV_Y +2 \sqrt 2 k \lambda f_{\gamma}V_\eta
\end{equation}
and
\begin{equation}\label{seconda-espressione-I-s1-sfera}
I_2 (fV_\eta)= \left ( \lambda^2-k^4 +2k^2\lambda\right )fV_\eta -2 \sqrt 2 k \lambda f_{\gamma} V_Y \,.
\end{equation}
\end{proposition}
The proof is based again on the general formula \eqref{I2-general-case}. The necessary calculations are entirely similar to those of Proposition\link\ref{proposizione-I-esplicito-toro-sfera} and so we omit the details.
Next, we decompose $\mathcal{C}\left( \varphi_k^{-1}T\s^2\right )$  in a similar fashion to \eqref{sottospazi-S-lambda}. We recall that the spectrum of $\Delta$ on $\s^1$ is $\{ m^2 \}_{m\in\n}$ and, for
$m \in \n$, we define
\begin{equation}\label{sottospazi-S-m}
 S^{m^2}=\left \{ f_1\,V_Y\,\,:\,\, \Delta f_1= m^2 f_1 \right \} \oplus \left \{ f_2\,V_\eta\,\,:\,\, \Delta f_2= m^2 f_2 \right \} \,.
\end{equation}
Then we know that $S^{m^2} \perp S^{m'^2}$ if $m \neq m'$, and $\oplus_{m=0}^{+\infty}\, S^{m^2}$ is dense in
$\mathcal{C}\left( \varphi_k^{-1}T\s^2\right )$. Moreover, $I_2$ preserves all these subspaces. Now, we observe that $\dim \left (S^0 \right )=2$ and that an orthonormal basis of $S^0$ is $\left\{u_1,u_2 \right\}$, where 
\[
u_1=  \frac{1}{\sqrt {2\pi}}\,V_Y,\,\quad
u_2= \frac{1}{\sqrt {2\pi}}\,V_\eta\,.
\]
Now, using \eqref{prima-espressione-I-s1-sfera} and \eqref{seconda-espressione-I-s1-sfera}, it is immmediate to construct the $(2 \times 2)$-matrix which describes the restriction of $I_2$ to $S^0$:
\begin{equation}\label{matrix-S-0}
\left (\begin{array}{rr}
0&0 \\
0&-k^4
\end{array}
\right )
\end{equation}
from which we deduce immediately that the contribution of $S^0$ to
the index and the nullity of $\varphi_k$ is $+1$ for both. Next, we study the subspaces $S^{m^2}$, $m \geq1$. First, we observe that $\dim \left (S^{m^2} \right )=4$ and that an orthonormal basis of $S^{m^2}$ is $\left\{u_1,u_2,u_3,u_4 \right\}$, where 
{\small \begin{equation}\label{on-bases}
u_1= \frac{1}{\sqrt {\pi}}\,\cos (m\gamma) \,V_Y,\;
u_2= \frac{1}{\sqrt {\pi}}\,\sin (m\gamma) \,V_Y,\;
u_3= \frac{1}{\sqrt {\pi}}\,\cos (m\gamma) \,V_\eta,\;
u_4= \frac{1}{\sqrt {\pi}}\,\sin (m\gamma) \,V_\eta\,.
\end{equation}
}
Now, using \eqref{prima-espressione-I-s1-sfera} and \eqref{seconda-espressione-I-s1-sfera}, we construct the $(4 \times 4)$-matrices which describe the restriction of $I_2$ to $S^{m^2}$. The outcome is:
\begin{equation}\label{matrici-I2-Sm}
\left(
\begin{array}{cccc}
 m^2 \left(3 k^2+m^2\right) & 0 & 0 & -2 \sqrt{2} k m^3 \\
 0 & m^2 \left(3 k^2+m^2\right) & 2 \sqrt{2} k m^3 & 0 \\
 0 & 2 \sqrt{2} k m^3 & m^4+2 m^2 k^2-k^4 & 0 \\
 -2 \sqrt{2} k m^3 & 0 & 0 & m^4+2 m^2 k^2-k^4
\end{array}
\right) \,,
\end{equation}
whose eigenvalues are
\begin{equation}\label{autovalori-I2-Sm}
\lambda_m^{\pm}=\frac{1}{2} \left(-k^4+2
   m^4+5 k^2 m^2\pm \sqrt{k^8+2 k^6 m^2+k^4 m^4+32 k^2m^6}\right)
\end{equation}
with multiplicity equal to $2$. Now, all the $\lambda_m^+$'s are clearly positive and so they do not contribute neither to the index nor to the nullity of $\varphi_k$. As for the $\lambda_m^-$'s, we can apply Lemma\link\ref{lemma-tecnico1}: it follows that the contribution to the nullity of $\varphi_k$ is $+2$ (coming from $\lambda_{k}^-$), while the contribution to the index is $+2(k-1)$, arising from $1 \leq m \leq  k-1$, so that the proof of Theorem\link\ref{Index-theorem-r-harmonic-circles} is completed.

\subsection{Proof of Theorem\link\ref{th: Sasahara_T2_S5}} The first part of the proof follows the lines of Theorem\link\ref{Spectrum-theorem-toro-sfera} and \cite{BFO}. For the sake of completeness and clarity, we report here the relevant facts. 
Let $ X_1=d\pi(\partial/\partial \gamma)$, $X_2=d\pi(\partial/\partial
\vartheta)$, where $\pi$ denotes the projection from $\mathbb{R}^2$ to
${\mathbb T}^2$. If $ U_1=d\varphi(X_1)$, $U_2=d\varphi(X_2)$, then
$\tau(\varphi)=-\phi(U_1)$. Moreover, the sections $U_1$, $U_2$,
$\phi(U_1)$, $\phi(U_2)$ and $\xi$ parallelize the pull-back
bundle $\varphi^{-1}T\mathbb{S}^5$.
As in \eqref{sottospazi-S-lambda}, we consider
$$
S^\lambda=\{fU_1\}_{f\in W_\lambda}\oplus\{fU_2\}_{f\in W_\lambda}
\oplus\{f\phi(U_1)\}_{f\in W_\lambda}
\oplus\{f\phi(U_2)\}_{f\in W_\lambda} \oplus\{f\xi\}_{f\in
W_\lambda},
$$
where $W_{\lambda}=\{f\in C^{\infty}({\mathbb T}^2): \Delta f=\lambda f\}$.
In this example, our torus is ${\mathbb T}^2=\s^1\times\s^1(1/\sqrt 2)$ and so the eigenvalues of its Laplace operator are of the form $\lambda=m^2+2n^2$, with $m,n \in \n$.
As above, $S^{\lambda_i} \perp S^{\lambda_j}$ if $i \neq j$ and $\oplus_{i=0}^{+\infty}\, S^{\lambda_i}$ is dense in
$\mathcal{C}\left( \varphi^{-1}T\s^5\right )$. 
The following version of Proposition\link\ref{proposizione-I-esplicito-toro-sfera} in this context was obtained in \cite{BFO}:
\begin{proposition}\label{prop-Legendre} Assume that $f \in W_{\lambda}$. Then
\begin{align*}
I_2(fU_1)=&(\lambda^2f-4X_2(X_2(f)))U_1-4X_1(X_2(f))U_2\\
&+4(\lambda+1)X_2(f)\phi(U_2)+(2\lambda f-4X_2(X_2(f)))\xi,\\
\mbox{}
I_2(fU_2)=&-4X_2(X_1(f))U_1+(\lambda^2+6\lambda)fU_2\\
& +4\lambda X_2(f)\phi(U_1)+
4(\lambda+1)X_1(f)\phi(U_2)
\\&-8X_1(X_2(f))\xi,
\end{align*}
\begin{align*}
I_2(f\phi(U_1))=&-4\lambda
X_2(f)U_2\\
&+(\lambda^2+4\lambda-4)f\phi(U_1)-8X_1(X_2(f))\phi(U_2)\\
&-4\lambda X_1(f)\xi,\\
\mbox{}
I_2(f\phi(U_2))=&-4(\lambda+1)X_2(f)U_1-4(\lambda+1)X_1(f)U_2\\
&-8X_1(X_2(f))\phi(U_1)+((\lambda^2+6\lambda)f-4X_2(X_2(f)))\phi(U_2)\\
&-4(\lambda+1)X_2(f)\xi,\\
\mbox{}
I_2(f\xi)=&(2\lambda f-4X_2(X_2(f)))U_1-8X_1(X_2(f))U_2\\
&+4\lambda X_1(f)\phi(U_1)+4(\lambda+1)X_2(f)\phi(U_2)\\
&+(\lambda^2+4\lambda)f\xi.
\end{align*}
\end{proposition}
As for $\lambda_0=0$, we have:
\begin{equation}\label{sottospazi-S-lambda-0-Legendre}
 S^{\lambda_0}=\left \{ c_1\,U_1 \right \} \oplus \left \{ c_2\,U_2 \right \}\oplus \left \{ c_3\,\phi(U_1) \right \}\oplus \left \{ c_4\,\phi(U_2) \right \}\oplus \left \{ c_5\,\xi \right \}\,,
 \end{equation}
where $c_i \in \R$, $\,i=1, \ldots, 5$, so that $\dim \left( S^{\lambda_0} \right)=5$.
Next, let us consider the case that $\lambda=m^2+2n^2 >0$. We decompose
\begin{equation}\label{sottospazi-W-lambda-Legendre}
W_{\lambda}=W^{m,0}\oplus_{m,n \geq 1}W^{m,n}\oplus W^{0,n}\,,
 \end{equation}
where it is understood that in \eqref{sottospazi-W-lambda-Legendre} we have to consider all the possible couples $(m,n)\in \n \times \n$ such that $\lambda=m^2+2n^2$. Now, the subspaces of the type $W^{m,0}$ are $2$-dimensional and are spanned by the functions $\left \{\cos (m\gamma),\sin (m\gamma) \right \}$. Similarly, $W^{0,n}$ is $2$-dimensional and is generated by $\left \{\cos (\sqrt 2 n\vartheta),\sin (\sqrt 2 n\vartheta) \right \}$. Finally, the subspaces $W^{m,n}$, with $m,n \geq 1$, have dimension $4$ and are spanned by $\left \{g_1,g_2,g_3,g_4 \right \}$, where
\begin{eqnarray}\label{definiz-g1-2-3-4}
&g_1=\cos (m\gamma)\cos(\sqrt 2 n\vartheta),\,\quad g_2=\cos (m\gamma)\sin (\sqrt 2 n\vartheta),\\\nonumber
&g_3=\sin (m\gamma)\cos(\sqrt 2 n\vartheta),\,\quad g_4=\sin (m\gamma)\sin (\sqrt 2 n\vartheta) \,.\nonumber
\end{eqnarray}
Now it becomes natural to define
\begin{equation}\label{sottospazi-S-m,n-Legendre}
 S^{m,n}=\left \{ f_1\,U_1\right \} \oplus \left \{ f_2\,U_2\right \} \oplus \left \{ f_3\,\phi(U_1)\right \}\oplus \left \{ f_4\,\phi(U_2)\right \}\oplus \left \{ f_5\,\xi\right \}\,,
 \end{equation}
where $f_i \in W^{m,n}$, $i=1, \ldots, 5$.
All these subspaces are orthogonal to each other. Moreover, for any positive eigenvalue $\lambda_i$, we have
\begin{equation}\label{decomposizioneS-lambda-in-coppie-m-n}
 S^{\lambda_i}= \oplus_{m^2+2n^2=\lambda_i}\,\,S^{m,n}\,.
 \end{equation}
Since $X_1$ and $X_2$ are Killing vector fields on ${\mathbb T}^2$ it follows easily from Proposition\link\ref{prop-Legendre} that the operator $I_2$ preserves each of the subspaces $S^{m,n}$. Therefore, its spectrum can be computed by determing the eigenvalues of the matrices associated to the restriction of $I_2$ to each of the $S^{m,n}\,$'s. The contribution to the index and the nullity of $\varphi$ arising from the subspaces $S^{\lambda_0}$, $S^{m,0}$, $S^{0,n}$, $S^{1,1}$ and $S^{2,1}$ has already been calculated in \cite {BFO}: it is $1+6+0+4+0=11$ for the index and $4+2+8+0+4=18$ for the nullity. By way of summary, in order to complete the proof of the theorem, we just have to study the subspaces $S^{m,n}$ in the remaining cases and show that they do not contribute neither to the index nor to the nullity of $\varphi$. To this purpose, we first observe that $\dim (S^{m,n})=20$ for all $m,n \in \n^*$ and an orthonormal basis for these subspaces is:
\begin{equation}\label{base-ortonormale-S-m-n}
\left \{c^*g_1U_1,c^*g_2 U_1,c^*g_3U_1,c^*g_4U_1,c^*g_1U_2,\ldots,c^*g_1\phi(U_1),\ldots, c^*g_1\phi(U_2),\ldots,c^*g_1\xi,\ldots\right \},
\end{equation}
where $g_1,g_2,g_3,g_4$ are the functions introduced in \eqref{definiz-g1-2-3-4} and $c^*=\sqrt[4]{2}/\pi$.
Next, using Proposition\link\ref{prop-Legendre}, a long but straightforward calculation leads us to the expression of the $20 \times 20$-matrices which describe the restriction of $I_2$ to the $S^{m,n}$'s with respect to the orthonormal bases \eqref{base-ortonormale-S-m-n}. The result is ($\lambda=m^2+2n^2$):
\begin{tiny}
\begin{equation*}
\setlength\arraycolsep{1.pt}
\left (
\begin{array}{cccccccccc}
 8 n^2+\lambda ^2 & 0 & 0 & 0 & 0 & 0 & 0 & -4 \sqrt{2} m n & 0 & 0   \cr
 0 & 8 n^2+\lambda ^2 & 0 & 0 & 0 & 0 & 4 \sqrt{2} m n & 0 & 0 & 0  \cr
 0 & 0 & 8 n^2+\lambda ^2 & 0 & 0 & 4 \sqrt{2} m n & 0 & 0 & 0 & 0  \cr
 0 & 0 & 0 & 8 n^2+\lambda ^2 & -4 \sqrt{2} m n & 0 & 0 & 0 & 0 & 0  \cr
 0 & 0 & 0 & -4 \sqrt{2} m n & \lambda  (\lambda +6) & 0 & 0 & 0 & 0 & -4 \sqrt{2} n \cr
 0 & 0 & 4 \sqrt{2} m n & 0 & 0 & \lambda  (\lambda +6) & 0 & 0 & 4 \sqrt{2} n \lambda  &0 \cr
 0 & 4 \sqrt{2} m n & 0 & 0 & 0 & 0 & \lambda  (\lambda +6) & 0 & 0 & 0 \cr
 -4 \sqrt{2} m n & 0 & 0 & 0 & 0 & 0 & 0 & \lambda  (\lambda +6) & 0 & 0   \cr
 0 & 0 & 0 & 0 & 0 & 4 \sqrt{2} n \lambda  & 0 & 0 & \lambda ^2+4 \lambda -4 & 0  \cr
 0 & 0 & 0 & 0 & -4 \sqrt{2} n \lambda  & 0 & 0 & 0 & 0 & \lambda ^2+4 \lambda -4   \cr
 0 & 0 & 0 & 0 & 0 & 0 & 0 & 4 \sqrt{2} n \lambda  & 0 & 0 \cr
 0 & 0 & 0 & 0 & 0 & 0 & -4 \sqrt{2} n \lambda  & 0 & 0 & 0  \cr
 0 & 4 \sqrt{2} n (\lambda +1) & 0 & 0 & 0 & 0 & 4 m (\lambda +1) & 0 & 0 & 0  \cr
 -4 \sqrt{2} n (\lambda +1) & 0 & 0 & 0 & 0 & 0 & 0 & 4 m (\lambda +1) & 0 & 0 \cr
 0 & 0 & 0 & 4 \sqrt{2} n (\lambda +1) & \;\;\;-4 m (\lambda +1) & 0 & 0 & 0 & 0 & 8 \sqrt{2} m
   n  \cr
 0 & 0 & -4 \sqrt{2} n (\lambda +1) & 0 & 0 & -4 m (\lambda +1) & 0 & 0 & -8 \sqrt{2} m n
   & 0  \cr
 2 \left(4 n^2+\lambda \right) & 0 & 0 & 0 & 0 & 0 & 0 & -8 \sqrt{2} m n & 0 & 0  \cr
 0 & 2 \left(4 n^2+\lambda \right) & 0 & 0 & 0 & 0 & 8 \sqrt{2} m n & 0 & 0 & 0  \cr
 0 & 0 & 2 \left(4 n^2+\lambda \right) & 0 & 0 & 8 \sqrt{2} m n & 0 & 0 & 4 m \lambda  & 0\cr
 0 & 0 & 0 & 2 \left(4 n^2+\lambda \right) & -8 \sqrt{2} m n & 0 & 0 & 0 & 0 & 4 m \lambda
  \end{array}
  \right.
  \end{equation*}

\begin{equation*}
\setlength\arraycolsep{-1pt}
\left .
\hspace{-5mm}\begin{array}{cccccccccc}
  0 & 0 & 0 & -4
   \sqrt{2} n (\lambda +1) & 0 & 0 & 2 \left(4 n^2+\lambda \right) & 0 & 0 & 0\cr
  0 & 0 & 4 \sqrt{2} n
   (\lambda +1) & 0 & 0 & 0 & 0 & 2 \left(4 n^2+\lambda \right) & 0 & 0 \cr
 0 & 0 & 0 & 0 & 0 &
   -4 \sqrt{2} n (\lambda +1) & 0 & 0 & 2 \left(4 n^2+\lambda \right) & 0 \cr
  0 & 0 & 0 & 0 & 4
   \sqrt{2} n (\lambda +1) & 0 & 0 & 0 & 0 & 2 \left(4 n^2+\lambda \right) \cr
 0 & 0 & 0 & 0 & -4 m (\lambda +1) & 0 & 0 & 0 & 0 & -8 \sqrt{2} m n \cr
  0 & 0 & 0 & 0 & 0 & -4 m (\lambda +1) & 0 & 0 & 8 \sqrt{2} m n & 0 \cr
 0 & -4 \sqrt{2}
   n \lambda  & 4 m (\lambda +1) & 0 & 0 & 0 & 0 & 8 \sqrt{2} m n & 0 & 0 \cr
  4 \sqrt{2} n
   \lambda  & 0 & 0 & 4 m (\lambda +1) & 0 & 0 & -8 \sqrt{2} m n & 0 & 0 & 0 \cr
  0 & 0 &
   0 & 0 & 0 & -8 \sqrt{2} m n & 0 & 0 & 4 m \lambda  & 0 \cr
  0 & 0
   & 0 & 0 & 8 \sqrt{2} m n & 0 & 0 & 0 & 0 & 4 m \lambda  \cr
 \lambda ^2+4 \lambda -4 & 0 &
   0 & 8 \sqrt{2} m n & 0 & 0 & -4 m \lambda  & 0 & 0 & 0 \cr
  0 & \lambda ^2+4 \lambda -4
   & -8 \sqrt{2} m n & 0 & 0 & 0 & 0 & -4 m \lambda  & 0 & 0 \cr
  0 & -8
   \sqrt{2} m n & \;\;\; 8 n^2+\lambda  (\lambda +6) & 0 & 0 & 0 & 0 & 4 \sqrt{2} n (\lambda +1)
   & 0 & 0 \cr
 8
   \sqrt{2} m n & 0 & 0 & 8 n^2+\lambda  (\lambda +6) & 0 & 0 & -4 \sqrt{2} n (\lambda +1)
   & 0 & 0 & 0\cr
  0 & 0 & 0 & 0 & 8 n^2+\lambda  (\lambda +6) & 0 & 0 & 0 & 0 & 4 \sqrt{2} n (\lambda
   +1) \cr
  0 & 0 & 0 & 0 & 0 & 8 n^2+\lambda  (\lambda +6) & 0 & 0 & -4 \sqrt{2} n (\lambda
   +1) & 0 \cr
  -4 m
   \lambda  & 0 & 0 & -4 \sqrt{2} n (\lambda +1) & 0 & 0 & \lambda  (\lambda +4) & 0 & 0 &
   0 \cr
  0 & -4 m
   \lambda  & 4 \sqrt{2} n (\lambda +1) & 0 & 0 & 0 & 0 & \lambda  (\lambda +4) & 0 & 0 \cr
 0 & 0 & 0 & 0 & 0 & -4 \sqrt{2} n (\lambda +1) & 0 & 0 & \lambda  (\lambda +4) & 0 \cr
  0 & 0 & 0 & 0 & 4 \sqrt{2} n (\lambda +1) & 0 & 0 & 0 & 0 & \lambda  (\lambda +4) 
  \end{array}
  \right)
  \end{equation*}
\end{tiny}
By means of a suitable software we find that their characteristic polynomial is:
\begin{equation}\label{pol-char-matrici20x20}
P(x)= \left[ P_5(x) \right ]^4 \,,
\end{equation}
where
\begin{equation}\label{pol-char-grado-5}
P_5(x)= a_5x^5+a_4x^4+a_3x^3+a_2x^2+a_1x+a_0
\end{equation}
with coefficients given by:
\begin{eqnarray}\label{coeff-pol-grado-5}
\\\nonumber a_5(m,n)&=&-1 \,\,;\\ \nonumber
a_4(m,n)&=&5 m^4+20 m^2 n^2+20 m^2+20 n^4+56 n^2-4\,\,; \\ \nonumber
a_3(m,n)&=& -10 m^8-80 m^6 n^2-48 m^6-240 m^4 n^4-320 m^4 n^2-96 m^4-320 m^2 n^6\\ \nonumber &&-704 m^2 n^4-272 m^2
   n^2+80 m^2-160 n^8-512 n^6-736 n^4+256 n^2\,\,;\\ \nonumber
a_2(m,n)&=&10 m^{12}+120 m^{10} n^2+24 m^{10}+600 m^8 n^4+240 m^8 n^2-8 m^8+1600 m^6 n^6\\ \nonumber &&+960 m^6
   n^4+1200 m^6 n^2-16 m^6+2400 m^4 n^8+1920 m^4 n^6+5024 m^4 n^4\\ \nonumber &&-320 m^4 n^2-320 m^4+1920
   m^2 n^{10}+1920 m^2 n^8+5440 m^2 n^6-2368 m^2 n^4\\ \nonumber &&-832 m^2 n^2+64 m^2+640 n^{12}+768
   n^{10}+512 n^8+512 n^6-2688 n^4+256 n^2 \,\,;\\ \nonumber
a_1(m,n)&=& -5 m^{16}-80 m^{14} n^2+16 m^{14}-560 m^{12} n^4+256 m^{12} n^2+64 m^{12}-2240 m^{10}n^6\\ \nonumber &&
   +1728 m^{10} n^4-560 m^{10} n^2+48 m^{10}-5600 m^8 n^8+6400 m^8 n^6-7072 m^8
   n^4\\ \nonumber &&+1536 m^8 n^2+272 m^8-8960 m^6 n^{10}+14080 m^6 n^8-23936 m^6 n^6+6272 m^6 n^4\\ \nonumber &&+6080
   m^6 n^2-64 m^6-8960 m^4 n^{12}+18432 m^4 n^{10}-34048 m^4 n^8+4608 m^4 n^6\\ \nonumber &&+12800 m^4
   n^4-256 m^4-5120 m^2 n^{14}+13312 m^2 n^{12}-18176 m^2 n^{10}\\ \nonumber &&-11520 m^2 n^8+45312 m^2
   n^6-5888 m^2 n^4+512 m^2 n^2-1280 n^{16}+4096 n^{14}\\ \nonumber &&-512 n^{12}-14336 n^{10}+18176
   n^8-4096 n^6-2048 n^4\,\,;\\ \nonumber
 \end{eqnarray}  
 \begin{eqnarray}\label{coeff-pol-grado-5-bis}
 &&\\
\nonumber a_0(m,n)&=&m^{20}+20 m^{18} n^2-12 m^{18}+180 m^{16} n^4-232 m^{16} n^2+44 m^{16}+960 m^{14} n^6\\ \nonumber &&-1984
   m^{14} n^4+656 m^{14} n^2-112 m^{14}+3360 m^{12} n^8-9856 m^{12} n^6+4832 m^{12}
   n^4\\ \nonumber &&+576 m^{12} n^2+304 m^{12}+8064 m^{10} n^{10}-31360 m^{10} n^8+22592 m^{10}
   n^6\\ \nonumber &&+11456 m^{10} n^4-5056 m^{10} n^2-320 m^{10}+13440 m^8 n^{12}-66304 m^8 n^{10}\\ \nonumber &&+70400
   m^8 n^8+44544 m^8 n^6-41152 m^8 n^4-1920 m^8 n^2+576 m^8+15360 m^6 n^{14}\\ \nonumber &&-93184 m^6
   n^{12}+144128 m^6 n^{10}+52992 m^6 n^8-116224 m^6 n^6-21504 m^6 n^4\\ \nonumber &&+3328 m^6 n^2-256
   m^6+11520 m^4 n^{16}-83968 m^4 n^{14}+184832 m^4 n^{12}\\ \nonumber &&-48128 m^4 n^{10}-121600 m^4
   n^8-22528 m^4 n^6+11264 m^4 n^4+1024 m^4 n^2\\ \nonumber &&+5120 m^2 n^{18}-44032 m^2 n^{16}+134144
   m^2 n^{14}-158720 m^2 n^{12}+54272 m^2 n^{10}\\ \nonumber &&-31744 m^2 n^8+44032 m^2 n^6-3072 m^2
   n^4+1024 n^{20}-10240 n^{18}+41984 n^{16}\\ \nonumber &&-98304 n^{14}+142336 n^{12}-124928
   n^{10}+60416 n^8-12288 n^6 \,\,. \nonumber
\end{eqnarray}
Now, the end of the proof of the theorem is an immediate consequence of the following technical lemma. To state the lemma, it is convenient to use the following notation:
\begin{definition} Let $m,n,m',n' \in \n$. We shall write $[m,n] \leq [m',n']$ if $m \leq m'$ and $n \leq n'$. We shall write $[m,n] < [m',n']$ if $[m,n] \leq [m',n']$ and either $m < m'$ or $n < n'$. 
\end{definition}
\begin{lemma}\label{lemma-pol-grado-5} Let $P_5(x)$ be the polynomial of degree $5$ defined in \eqref{pol-char-grado-5}--\eqref{coeff-pol-grado-5-bis}. Assume that 
\begin{equation}\label{uffa-uffa}
[2,1]<[m,n] \quad {\rm or} \quad [1,2]\leq [m,n]\,. 
\end{equation}
Then $P_5(x)$ does not admit any nonpositive root.
\end{lemma}
\begin{proof}[Proof of the lemma]
By the classical criterion of Descartes, it suffices to show that, if \eqref{uffa-uffa} holds, we have
\begin{equation}\label{criterio-cartesio}
{\rm (i)}\,\,a_5 < 0 \,; \quad {\rm (ii)}\,\,a_4 \geq 0 \,;\quad {\rm (iii)}\,\, a_3 \leq 0\,; \quad{\rm (iv)}\,\, a_2 \geq 0\,; \quad{\rm (v)}\,\,a_1 \leq 0\,; \quad{\rm (vi)}\,\, a_0 > 0\,.
\end{equation}
The claims \eqref{criterio-cartesio}(i),(ii) and (iii) are obvious. As for \eqref{criterio-cartesio}(iv), we first observe, by direct computation, that $a_2 \geq0$ when \eqref{uffa-uffa} holds and $[m,n] \leq [3,3]$. Then we rewrite the coefficient $a_2$ as follows:
\begin{eqnarray*}\label{riscrittura-a2}
a_2&=& 256 n^2+(-2688 n^4+512 n^6)+512 n^8+768 n^{10}+640 n^{12}+64 m^2\\ &&+(-2368 n^4 m^2+5440 n^6 m^2)+
1920 n^8 m^2+1920 n^{10} m^2\\ &&+(-320-320 n^2 +5024 n^4)m^4+1920 n^6 m^4+2400 n^8 m^4\\ &&+(-16 m^6+1200 n^2 m^6)+960 n^4 m^6+
1600 n^6 m^6\\ &&+(-832 n^2 m^2+240 n^2 m^8)+600 n^4 m^8+24 m^{10}+120 n^2 m^{10}+(-8 m^8+10 m^{12})
\end{eqnarray*}
Now it is easy to check that all the terms within parentheses are positive when $[3,3]<[m,n]$ and so \eqref{criterio-cartesio}(iv) is proved.
As for \eqref{criterio-cartesio}(v), we first observe, by direct computation, that $a_1 \leq 0$ when \eqref{uffa-uffa} holds and $[m,n] \leq [3,3]$. Then we rewrite the coefficient $a_1$ as follows:
\begin{eqnarray*}\label{riscrittura-a1}
a_1&=& -2048 n^4-4096 n^6+(18176 n^8-14336 n^{10})-512 n^{12}\\ &&+(4096 n^{14}-1280 n^{16})+(512 n^2 m^2-5888 n^4 m^2)\\ &&+
(45312 n^6 m^2-11520 n^8 m^2)-18176 n^{10 }m^2+(13312 n^{12} m^2-5120 n^{14} m^2)\\ &&-256 m^4+(12800 +4608 n^2-
34048 n^4)n^4 m^4+(18432 n^{10} m^4-8960 n^{12} m^4)\\ &&+(-64 +6080 n^2+6272 n^4-23936 n^6 )m^6\\ &&+(14080 n^8-8960 n^{10} )m^6+
(272 +1536 n^2 -7072 n^4)m^8\\ &&+(6400 n^6 m^8-5600 n^8 m^8)+(48 m^{10}-560 n^2 m^{10})\\ &&+(1728 n^4 m^{10}-2240 n^6 m^{10})+(64 m^{12}+
256 n^2 m^{12}-560 n^4 m^{12})\\ &&+(16 m^{14}-80 n^2 m^{14})-5 m^{16}
\end{eqnarray*}
Now it is not difficult to check that all the terms within parentheses are negative when $[3,3]<[m,n]$ and so \eqref{criterio-cartesio}(v) is proved.
As for \eqref{criterio-cartesio}(vi), we first observe, by direct computation, that $a_0 > 0$ when \eqref{uffa-uffa} holds and $[m,n] \leq [3,3]$. Then we rewrite the coefficient $a_0$ as follows:
\begin{eqnarray*}\label{riscrittura-a0}
a_0&=& (-12288 n^6+60416 n^8)+(-124928 n^{10}+142336 n^{12})\\ &&+(-98304 n^{14}+41984 n^{16})+(-10240 n^{18}+1024 n^{20})+
(-3072 n^4 m^2+44032 n^6 m^2)\\ &&+(-31744 n^8 m^2+54272 n^{10} m^2)+(-158720 n^{12} m^2+134144 n^{14} m^2)\\ &&+
(-44032 n^{16} m^2+5120 n^{18} m^2)+1024 n^2 m^4+11264 n^4 m^4\\ &&+(-48128 n^{10} m^4+184832 n^{12} m^4)+
(-83968 n^{14} m^4+11520 n^{16} m^4)\\ &&+
(-256 m^6+3328 n^2 m^6)+(-121600 n^8 m^4+52992 n^8 m^6)+144128 n^{10} m^6\\ &&+
(-93184 n^{12} m^6+15360 n^{14} m^6)+
576 m^8+(-41152 n^4 m^8+44544 n^6 m^8)\\ &&+(-116224 n^6 m^6+70400 n^8 m^8)+(-66304 n^{10} m^8+13440 n^{12} m^8)\\ &&+(-5056 n^2 m^{10}+
11456 n^4 m^{10})+(-22528 n^6 m^4+22592 n^6 m^{10})\\ &&+(-31360 n^8 m^{10}+8064 n^{10} m^{10})+304 m^{12}+
(-1920 n^2 m^8+576 n^2 m^{12})\\ &&+(-21504 n^4 m^6+4832 n^4 m^{12})+(-9856 n^6 m^{12}+3360 n^8 m^{12})\\ &&+
(-112 m^{14}+656 n^2 m^{14})+(-1984 n^4 m^{14}+960 n^6 m^{14})\\ &&+44 m^{16}+
(-232 n^2 m^{16}+180 n^4 m^{16} )+(-12 m^{18}+20 n^2 m^{18})+(m^{20} -320 m^{10})
\end{eqnarray*}
Now it is easy to check that all the terms within parentheses are positive when $[3,3]<[m,n]$ and so \eqref{criterio-cartesio}(vi) is proved and the proof of the lemma is ended.
\end{proof}
\begin{remark} It is easy to check, by direct inspection, that actually the inequalities \eqref{criterio-cartesio}(i)--(v) are true if \eqref{uffa-uffa} is replaced by the less restrictive condition $[1,1] \leq [m,n]$. By contrast, $a_0(1,1)<0$ and $a_0(2,1)=0$. Putting all these facts together we recover the result of \cite{BFO} according to which $S^{1,1}$ produces a contribution $4$ for the index and $0$ for the nullity, while the contribution of $S^{2,1}$ is $4$ for the nullity and $0$ for the index.
\end{remark}
\subsection{Proof of Theorem\link\ref{th:noncompact}} The first step of the proof is to compute the explicit expression of $I_2$ using \eqref{I2-general-case}. The image of $\varphi$ is contained in the equator of $\s^2$, which we denote by $\s^1$. Next, we define $Y$ and $\eta$ as follows:
\[
Y\left ( y_1,y_2,0\right )=\left ( -y_2,y_1,0\right )\quad {\rm and} \quad \eta \left ( y_1,y_2,0\right )=( 0,0,1)\,.
\]
We shall write $V_Y,\,V_\eta$ for $Y\circ \varphi,\,\eta\circ \varphi$ respectively. Clearly, $\left \{V_Y(\gamma),\,V_\eta(\gamma) \right \}$ is an orthonormal basis of $T_{\varphi(\gamma)}\s^2$ for all $\gamma \in \R$. Therefore, any section of $\varphi^{-1}T\s^2$ can be written as 
\[
V=f_1 \,V_Y+f_2 \,V_ \eta \,,
\]
where $f_1,\,f_2 \in C^{\infty}(\R)$.  Moreover,
$$
\tau(\varphi)=A''  \,V_Y\,, \quad \tau_2(\varphi)=A^{(4)}  \,V_Y\,.
$$
Next, performing computations similar to those of Proposition\link\ref{proposizione-I-esplicito-toro-sfera}, we compute the various terms of \eqref{I2-general-case}. The results are summarised in the following two lemmata.

\renewcommand{\arraystretch}{1.3}
\begin{lemma}\label{lemma13-noncompact}
\[\begin{array}{rcl}
\overline{\Delta}^2\left ( fV_Y\right ) &=& f^{(4)}V_Y\\ \nonumber
\overline{\Delta}\left ( {\rm trace}\langle fV_Y,d\varphi \cdot \rangle d\varphi \cdot- |d\varphi|^2\,fV_Y \right) &=& 0\\ \nonumber
2\langle d\tau(\varphi),d\varphi \rangle fV_Y &=& 2\, A'''\,A'\,f\,V_Y\\ \nonumber
|\tau(\varphi)|^2 fV_Y &=&(A'')^2\, f\,V_Y\\ \nonumber
-2\,{\rm trace}\langle fV_Y,d\tau(\varphi) \cdot \rangle d\varphi \cdot &=& -2\, A'''\,A'\,f\,V_Y\\ \nonumber
-2 \,{\rm trace} \langle \tau(\varphi),d(fV_Y) \cdot \rangle d\varphi \cdot &=&- 2\,
A''\,A'\,f' \,V_Y\\ \nonumber
- \langle \tau(\varphi),fV_Y \rangle \tau(\varphi)&=&-(A'')^2\, f\,V_Y \\ \nonumber
{\rm trace} \langle d\varphi \cdot,\overline{\Delta}(fV_Y) \rangle d\varphi \cdot &=&-(A')^2\, f''\,V_Y\\ \nonumber
{\rm trace}\langle d\varphi \cdot,\left ( {\rm trace}\langle
fV_Y,d\varphi \cdot \rangle d\varphi \cdot \right ) \rangle d\varphi \cdot&=& (A')^4\, f\,V_Y\\ \nonumber
-2 |d\varphi|^2\, {\rm trace}\langle d\varphi \cdot,fV_Y \rangle
 d\varphi \cdot&=&-2 (A')^4\, f\,V_Y \\ \nonumber
 2 \langle d(fV_Y),d\varphi\rangle \tau(\varphi) &=&
2 A''\,A'\,f'\,V_Y\\ \nonumber
-|d\varphi|^2 \,\overline{\Delta}(fV_Y)&=& (A')^2\,f''\,V_Y\\ \nonumber
|d\varphi|^4 fV_Y&=& (A')^4\,f\,V_Y\nonumber
\end{array}
\]
\end{lemma}
\begin{lemma}\label{lemma14-noncompact}
\[\begin{array}{rcl}
\overline{\Delta}^2\left ( fV_\eta\right ) &=& f^{(4)}V_\eta\\ \nonumber
\overline{\Delta}\left ( {\rm trace}\langle fV_\eta,d\varphi \cdot \rangle d\varphi \cdot- |d\varphi|^2\,fV_\eta \right) &=& \left [ 2 (A'')^2\,f+2\,A'''\,A'\,f+4\,A''\,A'\,f'+(A')^2 \,f''\right ]V_\eta\\ \nonumber
2\langle d\tau(\varphi),d\varphi \rangle fV_\eta &=& 2\, A'''\,A'\,f\,V_\eta\\ \nonumber
|\tau(\varphi)|^2 fV_\eta &=&(A'')^2\, f\,V_\eta\\ \nonumber
-2\,{\rm trace}\langle fV_\eta,d\tau(\varphi) \cdot \rangle d\varphi \cdot &=& 0\\ \nonumber
-2 \,{\rm trace} \langle \tau(\varphi),d(fV_\eta) \cdot \rangle d\varphi \cdot &=&0\\ \nonumber
- \langle \tau(\varphi),fV_\eta \rangle \tau(\varphi)&=&0 \\ \nonumber
{\rm trace} \langle d\varphi \cdot,\overline{\Delta}(fV_\eta) \rangle d\varphi \cdot &=&0\\ \nonumber
{\rm trace}\langle d\varphi \cdot,\left ( {\rm trace}\langle
fV_\eta,d\varphi \cdot \rangle d\varphi \cdot \right ) \rangle d\varphi \cdot&=& 0
\\\nonumber
-2 |d\varphi|^2\, {\rm trace}\langle d\varphi \cdot,fV_\eta \rangle
 d\varphi \cdot&=&0\\ \nonumber
 2 \langle d(fV_\eta),d\varphi\rangle \tau(\varphi) &=&0\\ \nonumber
-|d\varphi|^2 \,\overline{\Delta}(fV_\eta)&=& (A')^2\,f''\,V_\eta\\ \nonumber
|d\varphi|^4 fV_\eta&=& (A')^4\,f\,V_\eta\nonumber
\end{array}
\]
\end{lemma}
\renewcommand{\arraystretch}{1.}
Next, we insert the results given in Lemmata\link\ref{lemma13-noncompact} and  \ref{lemma14-noncompact} into \eqref{I2-general-case}. Then, adding up all the terms,  we obtain:
\begin{equation}\label{prima-espressione-I-noncompact}
I_2 (f\,V_Y)=f^{(4)}\,V_Y 
\end{equation}
and
\begin{equation}\label{seconda-espressione-I-noncompact}
I_2 (f\,V_\eta)= \left [f^{(4)}+2(A')^2\,f''+4A''\,A'\,f'+\left(4 A'''\,A'+3(A'')^2+(A')^4\right)\,f \right ]\,V_\eta \,.
\end{equation} 
Now, let $V=f_1 \,V_Y+f_2 \,V_ \eta\,$, assume that $f_1,\,f_2$ have compact support and denote by $D_V$ the support of $V$. Using \eqref{prima-espressione-I-noncompact} and \eqref{seconda-espressione-I-noncompact}, we find:
\begin{eqnarray}\label{IV,V-noncompact}
\int_{\R} \langle I_2(V),V \rangle \, d\gamma= \int_{D_V} &&\Big [f_1^{(4)}\,f_1 +f_2^{(4)}\,f_2+ \left((A')^2 f_2\right)'' f_2+(A')^2 f_2'' f_2
 \\ \nonumber
&& +\left(2 A'''\,A'+(A'')^2+(A')^4\right)\,f_2^2\Big ] d\gamma
\end{eqnarray}
Since $f_1,\,f_2$ and all their derivatives vanish on the boundary of $D_V$, performing suitable partial integrations it is easy to verify that \eqref{IV,V-noncompact} can be rewritten in the following more convenient way:
\begin{eqnarray}\label{IV,V-noncompact-bis}
\int_{\R} \langle I_2(V),V \rangle \, d\gamma= \int_{D_V} &\Big [(f_1'')^2 +\left(f_2''+(A')^2\,f_2\right)^2\\ \nonumber
& +\left((A'')^2+2A'''\,A'\right)\,f_2^2\Big ] d\gamma
\end{eqnarray}
It follows immediately that a sufficient condition to ensure that $\varphi$ is strictly stable is 
\begin{equation}\label{cond-suff-noncompact}
(A'')^2+2A'''\,A' \geq 0 \quad {\rm on} \,\,\R\,.
\end{equation}
Finally, a routine verification, using the explicit expression \eqref{A-gamma}, shows that \eqref{cond-suff-noncompact} is equivalent to \eqref{condiz-noncompact} and so the proof of Theorem\link\ref{th:noncompact} is completed.
\begin{remark}
We point out that in \eqref{IV,V-noncompact-bis} the quantity $(f_1'')^2+\left(f_2''+(A')^2\,f_2\right)^2$ is exactly the square of the norm of $J(V)$, where $J$ is the classical Jacobi operator.
\end{remark}
\begin{remark}
If \eqref{condiz-noncompact} is not satisfied, then $\varphi$ may be unstable. To see this, at least in the class of $C^5$-differentiable functions, we consider the case that $a=1, \,c=-2,\,b=d=0$, i.e., $A(\gamma)=\gamma^3-2\gamma$. We choose $f_1\equiv 0$ and
\[
f_2(\gamma)=\cos^6 \gamma \quad {\rm if} \,-\,\frac{\pi}{2} \leq \gamma \leq \frac{\pi}{2}, \,\, {\rm and}\,\,f_2(\gamma)=0 \,\, {\rm elsewhere} \,.
\] 
Then, replacing into \eqref{IV,V-noncompact-bis}, we find: 
\begin{eqnarray}\label{controesempio-noncompact}\nonumber
\int_{\R} \langle I_2(V),V \rangle \, d\gamma=& &\int_{-\pi/2}^{\pi/2} \Big [\left(36 \gamma^2+12 \left(3 \gamma^2-2\right)\right) \cos ^{12}\gamma\\ 
&& +\left(\left(3 \gamma^2-2\right)^2 \cos
   ^6\gamma-6 \cos ^6\gamma+30 \sin ^2\gamma \cos ^4\gamma\right)^2\Big ] d\gamma \simeq -3.537 <0\,.
\end{eqnarray}
\end{remark}
\section{Reduced Index and Nullity}\label{variation-of-example-toro-sfera}
A natural continuation of the study that we have undertaken in the previous sections would be to investigate the following family of equivariant maps:
\begin{align}\label{mappegeneralizzanti-toro-sfera}
\varphi_\alpha \,:\, \s^{n-1} (R) \times \s^1 & \to \s^n \hookrightarrow  \R^n \quad \times \R \\ \nonumber
(R\,\underline{\gamma},\qquad \vartheta)& \mapsto \, \left ( \sin \alpha(\vartheta) \, \underline{\gamma}, \,\cos  \alpha (\vartheta) \right ) \quad (\underline{\gamma} \in \s^{n-1},\,n \geq 2) \,.
\end{align}
The bienergy of a map as in \eqref{mappegeneralizzanti-toro-sfera} depends only on the function $\alpha$. Then it is natural to define the \textit{reduced bienergy} on $C^{\infty} \left(\s^1\right)$, as follows:
\begin{equation}\label{reduced-bienergy-generalizza-toro-sfera}
E_{2,{\rm red}}(\alpha)=E_2 \left (\varphi_\alpha \right )= \frac{1}{2}\, {\rm Vol}\left(\s^{n-1}(R) \right )\, \int_0^{2\pi } \left [\alpha'' - \frac{(n-1)}{2\,R^2}\sin (2\,\alpha) \right ]^2 d\vartheta
\end{equation}
The condition of biharmonicity for $\varphi_{\alpha}$ is the following ODE, which can be derived as the Euler-Lagrange equation of the reduced bienergy (see \cite{Mont-Ratto3}):
\begin{equation}\label{biarmoniatorosfera-generalizzato}
    \alpha^{(4)} - \alpha'' \, \left [ 2 \, \frac{(n-1)}{R^2} \, \cos (2 \alpha)\right ] + (\alpha')^2 \, \left [ 2 \, \frac{(n-1)}{R^2} \, \sin (2 \alpha)\right ] + \, \frac{(n-1)^2}{2 R^4} \, \sin (2\alpha) \, \cos (2 \alpha) \, = \, 0 \,
\end{equation}
and we still have the constant solution $\alpha \equiv \pi /4$. There are two basic differences between this example and the case of the maps which we have studied in Theorem\link\ref{Index-theorem-toro-sfera}: the higher dimension of the domain and the effects which derive from the fact that we admit a radius $R\neq1$. For these reasons, the computation of index and nullity becomes very complicated and therefore a rather natural approach in this context is to investigate what we shall refer to as the \textit{reduced index and nullity}. More precisely, let $\alpha^*$ denote the constant solution $\alpha\equiv \pi / 4$. We shall consider the reduced bienergy \eqref{reduced-bienergy-generalizza-toro-sfera} and its Hessian at the critical point $\alpha^*$
\begin{equation}\label{Hessian-definition-equiv}
H(E_{2,\rm red})_{\alpha^*} (V,W)= \left . \frac{\partial^2}{\partial t \partial s}\right |_{(0,0)}  E_{2,\rm red} (\alpha^*_{t,s}) \,,
\end{equation}
$\alpha^*_{t,s}$ being a two-parameter variation of $\alpha^*$ given by
\begin{equation}\label{equiv-2-par-variation}
\alpha^*_{t,s}(\vartheta)=\frac{\pi}{4}+t\, v(\vartheta)+s\, w(\vartheta) \,
\end{equation}
where $v,\,w \in C^{\infty}\left (\s^1 \right )$. As usually, the tangent vectors $V$ and $W$ to $C^{\infty}\left (\s^1 \right )$ at $\alpha^*$ are identified with $v$ and $w$, respectively.

Then, thinking $\s^n$ as the warped product $\left(\s^{n-1}\times [0,\pi], \sin^2\alpha \ g_{\s^{n-1}} + d\alpha^2\right)$,  we can identify $V$ and $W$ with the following sections of $\varphi_{\alpha^*}^{-1} T\s^n$
\[
V=\left .\frac{d}{d t} \right |_{t=0}\varphi_{\alpha^*_{t,0}}= v(\vartheta)\, \frac{\partial}{\partial \alpha} \quad {\rm and}\quad
W=\left .\frac{d}{d s} \right |_{s=0}\varphi_{\alpha^*_{0,s}}= w(\vartheta)\, \frac{\partial}{\partial \alpha} \,.
\]
From a geometric point of view, we observe that \eqref{Hessian-definition-equiv} is the restriction of the Hessian \eqref{Hessian-definition} to the following linear subspace of $\mathcal{C}\left(\varphi_{\alpha^*}^{-1} T\s^n\right)$:
\begin{equation}\label{sottospazio-equiv}
\mathcal{V}_{\rm red}=\left \{ V\in \mathcal{C}\left(\varphi_{\alpha^*}^{-1} T\s^n\right)\,\,:\,\,
V\left(\underline{\gamma},\vartheta \right )= v(\vartheta)\, \frac{\partial}{\partial \alpha},
\,\,v \in C^{\infty}\left (\s^1 \right )  \right \} \,,
\end{equation}
In particular, we observe that
\begin{equation}\label{I2-equiv}
H(E_{2,\rm red})_{\alpha^*} (V,W)=H(E_2)_{\varphi_{\alpha^*}}(V,W)={\rm Vol}\left(\s^{n-1} (R)\right )\, \int_0^{2\pi} \langle I_2(V),W \rangle d\vartheta \, .
\end{equation}
Moreover, the operator $I_2$ preserves the subspace $\mathcal{V}_{\rm red}$: this is a natural consequence of the symmetries of the problem and can be formally verified by a direct application of the general expression for $I_2$ (see \eqref{I2-general-case}). Therefore, the restriction of $I_2$ to $\mathcal{V}_{\rm red}$, which we shall denote $I_{2,{\rm red}}$, has a discrete spectrum and so we can define ${\rm Index}_{\rm red}(\varphi_{\alpha^*})$ and ${\rm Nullity}_{\rm red}(\varphi_{\alpha^*})$
precisely as in \eqref{Index-definition} and \eqref{Nullity-definition}. Our result in this context is the following:
\begin{proposition}\label{Index-theorem-equivariant}
Let $\varphi_{\alpha^*}\,:\,\s^{n-1}(R) \times \s^1 \to \s^n$ be the proper biharmonic map defined by \eqref{mappegeneralizzanti-toro-sfera} with $\alpha(\vartheta) \equiv \alpha^*$, where $\alpha^*=\pi /4$. If
\begin{equation}\label{condition-equiv-prop}
\frac{\sqrt {n-1}}{R} \not\in \n^* \,,
\end{equation}
then
\begin{align}\label{ind-null-equi}
& {\rm Nullity}_{\rm red}(\varphi_{\alpha^*})= 0 \\ \nonumber
& {\rm Index}_{\rm red}(\varphi_{\alpha^*})=1+2\,\left \lfloor
\frac{\sqrt {n-1}}{R}
\right \rfloor \,,\nonumber
\end{align}
where $\lfloor x\rfloor$ denotes the integer part of $x \in \R$.
If
\begin{equation}\label{condition-equiv-prop-bis}
\frac{\sqrt {n-1}}{R}\in \n^* \,,
\end{equation}
then
\begin{align}\label{ind-null-equi-bis}
& {\rm Nullity}_{\rm red}(\varphi_{\alpha^*})= 2 \\ \nonumber
& {\rm Index}_{\rm red}(\varphi_{\alpha^*})= 1+2\,\left (
\frac{\sqrt {n-1}}{R}
 -1 \right ) \,.\nonumber
\end{align}
\end{proposition}
\begin{remark}\label{remark-reduced-index}
Clearly, each eigenvalue $\lambda$ of $I_{2,{\rm red}}$ is also an eigenvalue of $I_2$ and $\mathcal{V}_{{\rm red},\lambda} \subseteq \mathcal{V}_{\lambda}$. Therefore, it is always true that ${\rm Index}_{\rm red}(\varphi) \leq {\rm Index}(\varphi)$ and ${\rm Nullity}_{\rm red}(\varphi) \leq {\rm Nullity}(\varphi)$. By way of example, if $\varphi_{k} : {\mathbb T}^2 \to \s^2$ is the proper biharmonic map defined in \eqref{equivdatoroasfera-bis*}, then it is not difficult to verify, using the technique of Proposition\link\ref{Index-theorem-equivariant}, that
\begin{align*}\label{ind-null-equi-bis}
& {\rm Nullity}_{\rm red}(\varphi_k)= 2 \\ \nonumber
& {\rm Index}_{\rm red}(\varphi_k)= 1+2\,(k-1) \,.\nonumber
\end{align*}
\end{remark}
%
\subsection{Proof of Proposition\link\ref{Index-theorem-equivariant}}
We compute the reduced Hessian \eqref{Hessian-definition-equiv} with respect to a two-parameter variation $\alpha^*_{t,s}$ as in \eqref{equiv-2-par-variation}. We obtain:
\begin{equation}\label{Hessian-equiv-explicit}
\left . \frac{\partial^2}{\partial t \partial s}\right |_{(0,0)}  E_{2,{\rm red}} (\alpha^*_{t,s})=c
\int_0^{2 \pi} \left [v''(\vartheta ) w''(\vartheta )-\frac{(n-1)^2 v(\vartheta ) w(\vartheta )}{R^4}\right] d\vartheta \,,
\end{equation}
where $c={\rm Vol}\left(\s^{n-1} (R)\right )\,$.
Comparing with \eqref{I2-equiv} and integrating by parts we find
\begin{equation}\label{I2-equiv-explicit}
I_{2,{\rm red}}(V)=\left [v^{(4)}(\vartheta )-\frac{(n-1)^2 }{R^4}\,v(\vartheta )\right ]\,\frac{\partial}{\partial \alpha} \,.
\end{equation}
Now, let
\begin{small}
\begin{equation}\label{base-equiv-case}
 \mathcal{B}= \left \{U_0= \frac{1}{\sqrt{2\pi}}\,\frac{\partial}{\partial \alpha} ,\,U_m= \frac{1}{\sqrt \pi}\,\cos (m\vartheta)\,\frac{\partial}{\partial \alpha} ,\,\,
V_m= \frac{1}{\sqrt \pi}\,\sin (m\vartheta)\,\frac{\partial}{\partial \alpha}, \,\,m \geq 1 \right \} \,.
 \end{equation}
\end{small}
Then $\mathcal{B}$ is an orthonormal basis of $\mathcal{V}_{\rm red}$. Moreover, it is immediate to check that the vectors $U_0,U_m$ and $V_m$ are eigenvectors for the operator \eqref{I2-equiv-explicit} with eigenvalues
\begin{equation}\label{eigenvalues-reduced}
\lambda_m=m^4-\frac{(n-1)^2 }{R^4}\,, \quad \,\, m \in \n .
\end{equation}\label{multiplicities-reduced}
Then the multiplicities are
\begin{equation}
\nu(\lambda_0)=1,\,\,\,\,\nu(\lambda_m)=2 \quad\,\,\, {\rm for} \,\, m \geq 1 \,,
\end{equation}
and now the conclusion of the proof follows easily.

\subsection{Conformal diffeomorphisms}\label{conformal-diffeo} Proper biharmonic conformal diffeomorphisms of $4$-dimensional Riemannian manifolds play an interesting role in the study of the bienergy functional. A basic example (see \cite{Baird}) is the inverse stereographic projection $\varphi:\R^4 \to \s^4$. We proved in \cite{MOR2} that its restriction to the open unit ball $B^4$ is strictly stable with respect to compactly supported equivariant variations. More generally, the same is true for homothetic dilations of $\varphi$ provided that we consider the restrictions to a ball whose radius ensures that the image is contained in the upper hemisphere of $\s^4$ (see \cite{MOR2} for details). 

Here we study another example which was found in \cite{Baird}. More precisely, using polar coordinates on $\R^4 \setminus\{ O \}$, let
\begin{align}\label{mappe-Baird-Fardoun}
\varphi_\alpha \,:\R^4\setminus \{O \}=& \s^{3} \times (0,+\infty)  \to \s^{3} \times \R \\ \nonumber
(&\underline{\gamma},\qquad r)\quad \quad  \mapsto \, \left (\underline{\gamma}, \,\alpha( r) \right )  \,,
\end{align}
where $\alpha(r)=\log r $. It was observed in \cite{Baird} that the map $\varphi_\alpha$ in \eqref{mappe-Baird-Fardoun} is a proper biharmonic conformal diffeomorphism. As in Theorem\link\ref{th:noncompact}, we study the stability of $\varphi_\alpha$ with respect to compactly supported variations. More precisely, since the domain here is not $1$-dimensional, we just consider equivariant variations:
\begin{align}\label{mappe-Baird-Fardoun-equiv-var}
\varphi_{\alpha_{t,s}} \,:\R^4\setminus\{O \}=& \s^{3} \times (0,+\infty)  \to \s^{3} \times \R \\ \nonumber
(&\underline{\gamma},\qquad r)\quad \quad  \mapsto \, \left (\underline{\gamma}, \,\alpha( r)+t v(r)+sw(r) \right )  \,,
\end{align}
where $v,w$ are compactly supported functions on $(0,+\infty)$. We shall prove the following result:
\begin{proposition}\label{prop-Baird-ex} Let $\varphi_\alpha \,:\R^4\setminus \{O \}  \to \s^{3} \times \R $ be the proper biharmonic conformal diffeomorphism defined in \eqref{mappe-Baird-Fardoun}. 
Then $\varphi_\alpha$ is strictly stable with respect to compactly supported equivariant variations.
\end{proposition}
\begin{proof}The tension field of a map of the type \eqref{mappe-Baird-Fardoun} is 
\[
\tau \left ( \varphi_\alpha\right )= \left [\alpha ''(r) + \frac{3}{r}\, \alpha'(r) \right ] \frac{\partial}{\partial \alpha}
\]
and so the reduced $2$-energy becomes:
\[
 E_{2,{\rm red}} (\alpha)=\frac{1}{2}\,c\, \int_0^{+\infty} \left [\alpha ''(r) + \frac{3}{r}\, \alpha'(r) \right ]^2\, r^3 \,dr \,,
\]
where $c={\rm Vol}\left ( \s^3 \right )$. It is convenient to make the following change of variable: $r=e^u$, $\beta(u)=\alpha(e^u)$. In terms of $\beta$, the reduced $2$-energy becomes:
\begin{equation}\label{Beta-2-energy}
 E_{2,{\rm red}} (\beta)=\frac{1}{2}\,c\,  \int_{-\infty}^{+\infty} \left [\beta ''(u) + 2\, \beta'(u) \right ]^2 \,du \,.
\end{equation}
Next, we compute the reduced Hessian \eqref{Hessian-definition-equiv} with respect to a two-parameter variation
\[
\beta_{t,s}= u + t v(u)+sw(u)\,,
\]
where $v,w$ are compactly supported smooth functions on $\R$ (note that, after the change of variable, $\alpha(r)= \log r$ corresponds to $\beta(u)=u$). We obtain:
\begin{equation}\label{Hessian-equiv-Baird}
\left . \frac{\partial^2}{\partial t \partial s}\right |_{(0,0)}  E_{2,{\rm red}} (\beta_{t,s})=c
\int_{-\infty}^{+\infty} \big [v''\,w''+2 v''\,w'+2 v'\, w''+4 v'\,
   w'\big] du \,.
\end{equation}
Now, since $v,w \in C_0^{\infty}(\R)$, integrating by parts into \eqref{Hessian-equiv-Baird} we find: 
\[
\int_{-\infty}^{+\infty} \,\langle I_{2,{\rm red}}(V),W \rangle \,du=\int_{-\infty}^{+\infty} \,\big[ v^{(4)}-4 v''\big] \,w \,\, du\,.
\]
Finally, again integrating by parts, we conclude that
\[
\int_{-\infty}^{+\infty} \,\langle I_{2,{\rm red}}(V),V \rangle \,du=\int_{-\infty}^{+\infty} \,\big[ (v'')^2+4 (v')^2\big]\,\, du \left(=\frac{1}{c} \int_{\r^4\setminus\{O\}} |J(V)|^2 dv_M\right)
\]
and so the proof is completed. 
\end{proof}
\subsection{Further developments}\label{Further-developments} A further generalisation of \eqref{mappegeneralizzanti-toro-sfera} leads us to study the case of biharmonic maps into the rotationally symmetric ellipsoid defined by
\begin{equation}\label{ellissoide}
\mathcal{Q} ^n(b)= \left \{ [x_1,\ldots,x_{n+1}]\in \R^{n+1}\,\,:\,\, x_1^2+\ldots +x_n^2+\frac{x_{n+1}^2}{b^2} =1\right \}\quad (b>0)\,.
\end{equation}
It is convenient to describe the ellipsoid $\mathcal{Q}^n(b)$ as
\begin{equation}\label{ellissoide-metrica}
\mathcal{Q} ^n(b)=\left ( \s^{n-1} \times [0,\pi],\sin^2 \alpha\,g_{\s^{n-1}}+K^2(\alpha)\,d\alpha^2 \right )\,,
\end{equation}
where $K(\alpha)=\sqrt{b^2\,\sin^2 \alpha+ \cos^2 \alpha}$.
Then we study equivariant maps
\begin{align}\label{mappegeneralizzanti-toro-ellissoide}
\varphi_\alpha \,:\, \s^{n-1} (R) \times \s^1 & \to \mathcal{Q}^n \hookrightarrow  \R^n \quad \times \R \\ \nonumber
(R\,\underline{\gamma},\qquad \vartheta)& \mapsto \, \left ( \sin \alpha(\vartheta) \, \underline{\gamma}, \,b\,\cos  \alpha (\vartheta) \right ) \quad (\underline{\gamma} \in \s^{n-1},\,n \geq 2) \,.
\end{align}
Now, the reduced bienergy of a map as in \eqref{mappegeneralizzanti-toro-ellissoide} is
\begin{equation}\label{reduced-bienergy-generalizza-toro-ellissoide}
E_{2,{\rm red}}(\alpha)= \frac{1}{2}\, {\rm Vol}\left(\s^{n-1}(R) \right )\, \int_0^{2\pi } \left [\alpha'' - \frac{(n-1)}{2\,R^2\,K^2(\alpha)}\sin (2\,\alpha)+\frac{K'(\alpha)}{K(\alpha)}\,(\alpha')^2 \right ]^2 K^2(\alpha) \,d\vartheta \,.
\end{equation}
By direct substitution into the Euler-Lagrange equation of \eqref{reduced-bienergy-generalizza-toro-ellissoide} we find that a constant function $\alpha(\vartheta)\equiv \alpha^*$ ($0< \alpha^*<\pi/2$) gives rise to a proper biharmonic map of the type \eqref{mappegeneralizzanti-toro-ellissoide} if and only if
\begin{equation}\label{biarmonia-ellissoide-alpha*}
\alpha^*=\frac{1}{2} \arccos \left [\frac{b-1}{b+1} \right ] \,.
\end{equation}
Following the method of proof of Proposition\link\ref{Index-theorem-equivariant} and computing we find that now the expression of the reduced index operator is:
\begin{equation}\label{I2-equiv-explicit-ellissoide}
I_{2,{\rm red}}(V)=\left [V^{(4)}-\frac{4 (n-1)^2}{b (b+1)^2 R^4}\,V\right ]\,\frac{\partial}{\partial \alpha} \,.
\end{equation}
Then it is easy to conclude that the eigenvalues of $I_{2,{\rm red}}$ are
\begin{equation*}\label{eigenvalues-reduced}
\lambda_m=m^4-\frac{4(n-1)^2 }{b(b+1)^2 R^4}\,, \quad \,\, m \in \n ,
\end{equation*}\label{multiplicities-reduced}
with multiplicities 
\begin{equation*}
\nu(\lambda_0)=1,\,\,\,\,\nu(\lambda_m)=2 \quad\,\,\, {\rm for} \,\, m \geq 1 \,.
\end{equation*}
By way of summary, we conclude that an extension of Proposition\link\ref{Index-theorem-equivariant} holds in this context where the curvature of the target is nonconstant. Indeed, we have
\begin{proposition}\label{Index-theorem-ellissoide} Let $\varphi_{\alpha^*}\,:\,\s^{n-1}(R) \times \s^1 \to \mathcal{Q}^n(b)$ be the proper biharmonic map defined by \eqref{mappegeneralizzanti-toro-ellissoide} with $\alpha(\vartheta) \equiv \alpha^*$, where $\alpha^*$ is given in \eqref{biarmonia-ellissoide-alpha*}. If
\begin{equation}\label{condition-equiv-prop-bis}
\sqrt[4]{\frac{4(n-1)^2 }{b(b+1)^2 R^4}} \not\in \n^* \,,
\end{equation}
then
\begin{align}\label{ind-null-equi-bis}
& {\rm Nullity}_{\rm red}(\varphi_{\alpha^*})= 0 \\ \nonumber
& {\rm Index}_{\rm red}(\varphi_{\alpha^*})=1+2\,\left \lfloor
\sqrt[4]{\frac{4(n-1)^2 }{b(b+1)^2 R^4}}
\right \rfloor \,,\nonumber
\end{align}
where $\lfloor x\rfloor$ denotes the integer part of $x \in \R$.
If
\begin{equation}\label{condition-equiv-prop-bis}
\sqrt[4]{\frac{4(n-1)^2 }{b(b+1)^2 R^4}}\in \n^* \,,
\end{equation}
then
\begin{align}\label{ind-null-equi-bis}
& {\rm Nullity}_{\rm red}(\varphi_{\alpha^*})= 2 \\ \nonumber
& {\rm Index}_{\rm red}(\varphi_{\alpha^*})= 1+2\,\left (
\sqrt[4]{\frac{4(n-1)^2 }{b(b+1)^2 R^4}}
 -1 \right ) \,.\nonumber
\end{align}
\end{proposition}
\begin{remark}\label{remark-ellissoide}
We observe that the value $\alpha^*$ in \eqref{biarmonia-ellissoide-alpha*} corresponds to the parallel proper biharmonic hypersphere in $\mathcal{Q}^n(b)$ which was found in \cite{Mont-Ratto2}. We also point out that, for fixed values of $n,R$, there exists $b^*>0$ such that if $ b \geq b^*$, then the reduced index is $1$. By contrast, the reduced index becomes arbitrarily large provided that $b$ is sufficiently small.
\end{remark}
A natural question is to ask whether variations associated to vector fields in the nullity subspace give rise to variations made of biharmonic maps. We have checked by direct substitution into \eqref{biarmoniatorosfera-generalizzato} that, in general, this is not the case. Actually, variations of this type do not even preserve the bienergy. For instance, assume $n=2$ and $R=b=1$. Then the variation
\[
\alpha_t(\vartheta)=\frac{\pi}{4}+t \sin \vartheta
\]
gives rise to a vector field in the nullity subspace and a direct computation shows that, up to a constant factor, 
\[
 \frac{d}{dt}\, E_{2,{\rm red}} (\alpha_{t})= \pi\, t -\frac{1}{2} \, \pi \, J_1(4t) \,,
\]
where $J_1$ denotes the Bessel function $J_n$ of the first type, with $n=1$ (see \cite{Bessel} for definitions and properties of Bessel functions). From this it is possible to deduce that
\begin{equation*}
\left . \frac{d^j}{dt^j}\right |_{t=0}  E_{2,{\rm red}} (\alpha_{t})=0\,,\quad \quad j=1,2,3. \qquad
\left . \frac{d^4}{dt^4}\right |_{t=0}  E_{2,{\rm red}} (\alpha_{t})=12 \pi \,.
\end{equation*}
In particular, along this direction in the nullity subspace we see that the bienergy has a local minimum.

\end{document}